\documentclass[10pt, reqno]{amsart}

\usepackage{amssymb,latexsym}
\usepackage{eucal}
\newtheorem{corollary}{Corollary}
\newtheorem{lemma}{Lemma}
\newtheorem{theorem}{Theorem}

\newcommand{\R}{\mathcal{R}}

\newtheorem{mydef}{Definition}

\newtheorem{problem}{Problem}
\newtheorem{observation}{Observation}

\title[Dehn Functions, the Word Problem, and the Bounded Word Problem]{Dehn Functions, the Word Problem, and the Bounded Word Problem For Decidable Group Presentations}

\author{Desmond Cummins}

\date{\today}

\begin{document}

\maketitle

\begin{abstract}

We construct examples of finitely generated decidable group presentations that satisfy certain combinations of solvability for the word problem, solvability for the bounded word problem, and computablity for the Dehn function.  We prove that no finitely generated decidable presentations exist satisfying the combinations for which we do not provide examples.  The presentations we construct are minimal.  These constructions answer an open question asked by R.I. Grigorchuk and S.V. Ivanov.

Our approach uses machinery developed by Birget, Ol'shanskii, Rips, and Sapir for constructing finite group presentations that simulate Turing machines.  We generalize this machinery to construct finitely generated decidable group presentations that simulate computing objects similar to oracle Turing machines.

\end{abstract}

\section{Introduction}

It is well known that the word problem is solvable for a finite presentation $P$ if and only if the Dehn function of $P$ is computable.  Moreover, the bounded word problem is in $\textbf{NP}$ for all finite presentations.  If $P$ is a finitely generated decidable presentation, it is not evident whether the solvability of the word problem is equivalent to the computability of the Denh function.  In \cite{GI}, the authors Grigorchuk and Ivanov pose the following problem:

\begin{problem} \label{prob}

\cite{GI} Let the relator set $\R$ of
a finitely generated presentation $P$ be decidable.
Prove or disprove that

$(\mathrm{a})$  If the word problem for $P$ is
solvable,
 then the Dehn function $f(x)$  of P  is
 computable.

$(\mathrm{b})$   If the Dehn function  $f(x)$  of
$P$ is computable,  then the word problem for
$P$ is solvable.
\end{problem}

A solution to Problem 1(b) is given via a counterexample in \cite[Example 2.4]{GI}.  It is also pointed out in \cite{GI} that it would be of interest to consider a stronger and presumably more difficult version of Problem 1 in which the relator set $\R$ is assumed to be {\em minimal} (i.e. for all $\R' \subseteq \R$, if the normal closures $\langle \langle \R'  \rangle \rangle$ and $\langle \langle \R  \rangle \rangle$ are equal, then $\R' = \R$).  Note that the presentation given as counter example to Question 1(b) in \cite[Example 2.4]{GI} is not minimal.\\

We further expand Question \ref{prob} to consider the solvability of the bounded word problem as well.  There are eight possibilities for the solvability of the word problem, the solvability of the bounded word problem, and the computability of the Dehn function, as shown in the following table.  For example, case \textbf{2} refers to the possibility of a presentation having solvable word problem and solvable bounded word problem, but uncomputable Dehn function.  It is straightforward to prove that finitely generated presentations satisfying cases $\textbf{2}$ and $\textbf{5}$ do not exist (see Lemma \ref{no2no5} in Section 3).  The purpose of this paper is to construct finitely generated decidable minimal presentations to satisfy each of the remaining cases.\\

\begin{center}

\begin{tabular}{|l||c|c|c|c|c|c|c|r|}

\hline

Case Number                      & \textbf{1} & \textbf{2} & \textbf{3} & \textbf{4} & \textbf{5} & \textbf{6} & \textbf{7} & \textbf{8} \\ \hline

Solvable Word Problem            & y & y & y & y & n & n & n & n \\ \hline

Solvable Bounded Word Problem    & y & y & n & n & y & y & n & n \\ \hline

Computable Dehn Function         & y & n & y & n & y & n & y & n \\

\hline

\end{tabular}

\end{center}

\vspace{.5cm}

There has been a significant amount of previous research done that involves constructing group presentations that, in some sense, simulate Turing machines.  One such result is the classical Novikov-Boone-Higman-Aanderaa embedding of a finitely generated group $G$ into a finitely presented group $H$ \cite{Rot}. In \cite{MO}, Madlener and Otto explored the idea of constructing such an embedding so that the Dehn function of the presentation for $H$ was not ``much bigger" than the time complexity of the Turing machine it simulates.  In \cite{BRS} and \cite{BORS}, Birget, Ol'shanski, Rips, and Sapir construct such an embedding in which the Dehn function of the presentation $H$ is equivalent to the fourth power of the time function of the Turing machine (provided the fourth power of time function is superadditive).  In these papers, the authors prove the remarkable Theorem \ref{bigdeal}, which will be our primary tool in this paper.  See Section 2 for the definition of {\em superadditive}.

\begin{theorem}\label{bigdeal}

Let $L\subseteq A^*$ be a language accepted by a Turing machine $M$ with time function $T(n)$ for which $T(n)^4$ is superadditive.  Then there exists a finite group presentation $P(M)$ with generating set $X$ and with Dehn function equivalent to $T(n)^4$.  Also, there exists an injective map $\mathcal{K}:A^*\rightarrow (X\cup X^{-1})^*$ such that

\begin{enumerate}

\item
$\textbf{u}\in L$ if and only if $\mathcal{K}(\textbf{u})=1$ in $P(M)$;

\item
$\mathcal{K}(\textbf{u})$ has length $O(|\textbf{u}|)$.  There is a linear-time algorithm that takes as input a word $\textbf{u}$ in $A^*$ and outputs $\mathcal{K}(\textbf{u})$.

\item
If $\textbf{u}\in L$, and $\ell(\textbf{u})$ is the length of a minimal length accepting $M$ computation for $\textbf{u}$, then the minimal area $P(M)$ diagram with boundary label $\mathcal{K}(\textbf{u})$ has area equal to $O(\ell(\textbf{u})^4)$.

\end{enumerate}

\end{theorem}

Parts 1 and 2 of the above theorem are exactly as in\cite[Theorem 1.3]{BRS}.  Part 3 of the above theorem is a direct consequence of part 3 of \cite[Proposition 4.1]{BRS} and the proof of \cite[Proposition 12.1]{BRS}.  \\

Note that in the statement of \cite[Theorem 1.3]{BRS}, the authors do not mention the presentation $P(M)$.  Instead, they phrase the result as a statement about the group $G(M)$ which has $P(M)$ as its presentation.  The reason for our departure from this notation of \cite{BRS} is that later in this paper we will be using results analogous to Theorem \ref{bigdeal} to construct decidable presentations with certain Dehn functions.  Unlike in the case of finite presentations, distinct decidable presentations of the same group do not necessarily have equivalent Dehn functions.  For this reason we find it necessary to write all our results as statements about group presentations instead of statements about the groups they present.  In keeping with this convention, we will often phrase statements about groups as statements about the presentations of those groups.  For example, if a group $G$ has presentation $P$, we will write ``the word problem is solvable for $P$" instead of ``the word problem is solvable for $G$".  Also, if $G'$ is a group with presentation $P'$, we will write ``there is an embedding of $P$ into $P'$" to indicate that there is an embedding of $G$ into $G'$.  Additionally, we will write ``$w$ is a word in $P$" to indicate that $w$ is a word in the generators of $P$.\\

Theorem \ref{bigdeal} is not immediately applicable to our problem since the presentation $P(M)$ is finite.  The finiteness of $P(M)$ is due to the fact that the size of the sets of generators and relators of $P(M)$ depends on the size of the sets of commands, states, and alphabet letters of $M$.  Each of these sets is finite.  Our constructions will involve generalizing the methods of \cite{BRS} to construct group presentations from countably infinite objects that are analogous to oracle Turing machines.  We will call these objects {\em union machines}.\\

In Section 2, we will formally define union machines. Intuitively, a union machine $M_{\infty}$ is a tuple of sets that satisfies every condition in the definition of a Turing machine with the one exception that the some of the sets that constitute $M_{\infty}$ may be countably infinite. If the tuple of sets that constitutes a union machine $M_{\infty}$ are all computably enumerable (or c.e.), then $M_{\infty}$ is said to be c.e..\\

It happens that the proof of Theorem \ref{bigdeal} in \cite{BRS} actually proves a more general result than Theorem \ref{bigdeal}.  This is because neither the construction of $P(M)$ from $M$ nor the proof of Theorem \ref{bigdeal} rely in any way on the finiteness of the sets that constitute $M$.   Given a union machine $M_{\infty}$, not only can we construct a group presentation $P(M_{\infty})$ from $M_{\infty}$ in the exact same way that $P(M)$ is constructed from $M$, but a proof identical to that of Theorem \ref{bigdeal} in \cite{BRS} proves the following result about $P(M_{\infty})$.\\

\begin{theorem} \label{smalldeal}

Let $L\subset A^*$ be a language accepted by a c.e. union machine $M_{\infty}$ with time function $T(n)$ such that $T(n)^4$ is equivalent to a superadditive function.  Then there exists a countably generated minimal c.e. group presentation $P(M_{\infty})$ with generating set $X$ and with Dehn function equivalent to $T(n)^4$.  Also, there exists an injective map $\mathcal{K}:A^*\rightarrow (X\cup X^{-1})^*$ such that

\begin{enumerate}

\item
$\textbf{u}\in L$ if and only if $\mathcal{K}(\textbf{u})=1$ in $P(M_{\infty})$;\\

\item
$\mathcal{K}(\textbf{u})$ has length $O(|\textbf{u}|)$.  There is a linear-time algorithm that takes as input a word $\textbf{u}$ in $A^*$ and outputs $\mathcal{K}(\textbf{u})$.\\

\item
If $\textbf{u}\in L$, and $\ell(\textbf{u})$ is the length of a minimal length accepting $M_{\infty}$ computation for $\textbf{u}$, then the minimal area $P(M_{\infty})$ diagram with boundary label $\mathcal{K}(\textbf{u})$ has area equal to $O(\ell(\textbf{u})^4)$.

\end{enumerate}

\end{theorem}

There are three minor details of Theorem \ref{smalldeal} that must be addressed in order to use the arguments of \cite{BRS} as a proof of Theorem \ref{smalldeal}.  The first such detail is the claim in Theorem \ref{smalldeal} that if $M_{\infty}$ is c.e. then $P(M_\infty)$ is c.e..  Since \cite{BRS} only deals with Turing machines and finite presentations, this claim is not explicitly mentioned anywhere in \cite{BRS}.  However, it follows immediately from the constructions in \cite{BRS} that $P(M_{\infty})$ is c.e. if $M_{\infty}$ is c.e..  We provide some citation for this in Observation \ref{symce} and in Lemmas \ref{smachce} and \ref{presce}.\\

Secondly, Theorem \ref{bigdeal} makes no mention of the minimality of the presentation $P(M)$.  However, both $P(M)$ and $P(M_{\infty})$ are minimal.  This is a straightforward consequence of the structure of the HNN extensions used to construct $P(M)$ and $P(M_{\infty})$, as is explained further in the proof of Lemma \ref{minimal1}.\\

Finally, note that we have weakened the requirement in Theorem \ref{bigdeal} that $T^4(n)$ must be superadditive to the requirement in Theorem \ref{smalldeal} that $T^4(n)$ must be equivalent to a superadditive function.  We can do this because proof of Theorem \ref{bigdeal} in \cite{BRS} only uses superadditivity of $T^4$ up to equivalence, and does not rely on the actual superadditivity of $T^4$. This can be quickly verified by reading page 355 of \cite{BRS}, where the authors give a brief summary of the proof of Theorem \ref{bigdeal}.\\

It should be noted that Dehn functions of countably generated presentations are not always as well behaved as those of finite presentations.  In particular, they are not always well defined \cite{Os}.  This is not a concern for the presentation $P(M_{\infty})$, whose Dehn function is guaranteed to be well defined by Theorem \ref{smalldeal}.\\

Our desired examples of group presentations are decidable, minimal, and finitely generated.  We also require some control of the solvability of the word problem for our presentations.  We will use Theorem \ref{smalldeal} to prove the following theorem.\\

\begin{theorem}\label{main}

Let $L\subset A^*$ be a language accepted by a c.e. union machine $M_{\infty}$ with time function $T(n)$ such that $T(n)^4$ is superadditive.  There exists a finitely generated decidable minimal presentation $P'_1(M_{\infty})$ with generating set $\{a,b\}$ that has the following properties.  There exists an injective map $h\circ \mathcal{K}:A^*\rightarrow \{a^{\pm1},b^{\pm1}\}^*$ such that:

\begin{enumerate}

\item
For an input word $\textbf{u}$ of $M_{\infty}$, the word $h(\mathcal{K}(\textbf{u}))$ is trivial in $P'_1(M_{\infty})$ if and only if $\textbf{u}$ is an acceptable input of $M_{\infty}$.

\item
The word problem for $P'_1(M_{\infty})$ is solvable if and only if $L$ is decidable.

\item
$P'_1(M_{\infty})$ has Dehn function equivalent to $T(n)^4$.

\item
Suppose $\textbf{u}\in L$, and a minimal length accepting $M_{\infty}$ computation for $\textbf{u}$ has length $\ell(\textbf{u})$.  Then the minimal area $P'_1(M_{\infty})$ diagram with boundary label $h(\mathcal{K}(\textbf{u}))$ has area equal to $O(\ell(\textbf{u})^4)$.

\end{enumerate}

\end{theorem}

In Section 2 we give the formal definitions of our objects of study (Turing machines, group presentaitons, etc.).  In Section 3, we will show how to use Theorem \ref{main} to construct the desired examples of group presentations.  The rest of the paper will be devoted to proving Theorem 3.\\

This proof will consist of two main parts.  The first part is proving that the word problem for $P(M_{\infty})$ is solvable if and only if the language accepted by $M_{\infty}$ is decidable.  This is done in three stages.  The first stage is proving that the symmetrization $M'_{\infty}$ of $M_{\infty}$ has solvable configuration problem if and only if the language accepted by $M_{\infty}$ is decidable.  This is done in Section 4.  The second stage is proving that $M'_{\infty}$ has solvable configuration problem if and only if the $S$-machine $S(M'_{\infty})$ has solvable configuration problem.  This is done in Section 5.  The third stage is proving that the word problem for $P(M_{\infty})$ is solvable if and only if the configuration problem for $S(M'_{\infty})$ is solvable.  This is done in Section 8.\\

The second part uses Theorem \ref{smalldeal} to complete the proof of Theorem 3.  This proof will involve constructing the finitely generated presentation $P'_1(M_{\infty})$ from $P(M_{\infty}$ in) such a way that $P'_1(M_{\infty})$ inherits desired properties from $P(M_{\infty})$.  This construction will be nearly identical to a similar construction performed in \cite{FG}.  It will have the additional fortunate property that $P'_1(M_{\infty})$ is decidable if $P(M_{\infty})$ is c.e..  Conveniently, several of the lemmas required for this part are already proven in \cite{FG}.  This will part will be done in Section 9.\\

To keep this paper at a manageable length, we have cited the results and proofs from \cite{BRS}, \cite{BORS}, and \cite{FG} wherever possible.  In order to make these citations easily verifiable, we have made an effort to keep our definitions and notations as close as possible those found in \cite{BRS}, \cite{BORS}, and \cite{FG}.  Also, since \cite{BRS} and \cite{BORS} are not short, we have provided page numbers along with citations whenever is seemed appropriate.\\

\section{Preliminaries and Definitions}

If $A$ is a set of symbols, we use $A^*$ to denote the set of finite sequences of symbols of $A$ (including the empty sequence, denoted $\varepsilon$).  If $L\subseteq A^*$, then we call $L$ a {\em language} over $A$.  For $L\subseteq A$, we say that $L$ is {\em decidable} if there is a Turing machine $M$ that accepts an input $\textbf{u}\in A^*$ if $\textbf{u}\in L$ and halts in a non-accept state $\textbf{u}$ if $\textbf{u}\notin L$.  We say that $L$ is {\em computably enumerable} or {\em c.e.} if there is a Turing machine $M$ that accepts an input $\textbf{u}\in A$ if and only if $\textbf{u}\in L$ (note that if $\textbf{u}$ is not in $L$ then $M$ need not halt on input $\textbf{u}$).  Intuitively, a set is c.e. if it is possible to algorithmically produce a list of its elements.  If a process or a construction can be performed algorithmically, we say that the process or construction is {\em effective}.\\

A group presentation is a set of generators and a set of defining relators

\begin{equation}\label{pr3}
P  = \langle \ X \ \| \ R  \ \rangle \ ,
\end{equation}

\smallskip

\noindent where $X = \{ a_1, a_2, \dots \}$ is a countable alphabet and $R$ is a set of nonempty cyclically reduced words over the alphabet $X^{\pm 1} = X \cup X^{-1}$ (we assume that $R$ is closed under taking inverses and cyclic conjugates).  Let $F(X)$ denote the free group over $X$, $|w|$ denote the length of an element $w \in F(X)$, and $\langle\langle R \rangle\rangle $ denote the normal closure of $R$ in $F(X)$.  The presentation \eqref{pr3} denotes the quotient group $G=F(X)/\langle\langle R \rangle\rangle $. Recall that a presentation \eqref{pr3} is called {\em finite} if $X$ and $R$ are finite.\\

A function $f: \textbf{N} \to \textbf{N}$ is called an {\em isoperimetric function} of a presentation $P=\langle X \| R\rangle$ if for every number $n$ and every word $w$ trivial in $P$ with $|w| < n$, there exists a van Kampen diagram over $P$ with boundary label $w$ and area $< f(n)$  (or, equivalently, $w$ is a product of at most $f(n)$ conjugates of the relators from $R$, see \cite{MO}, \cite{Gr2}, \cite{Gersten}, \cite{BMS}).  The smallest isoperimetric function of a decidable presentation $P$ is called the {\em Dehn function} of $P$.\\

A presentation $P=\langle X \| R\rangle$ is {\em minimal} if if fulfils the following condition.  For a relator set $R'$ over $X$ with $R'\subseteq R$, if $R$ and $R'$ have the same normal closure, then $R=R'$.\\

Let $f, g : \textbf{N} \to \textbf{N}$ be two functions.  We write $f\preceq g$ if there exists a nonnegative constant $d$ such that $f(n) < dg(dn) + dn$.  As in \cite{BRS}, all functions $g(n)$ that are considered in this paper grow at least as fast as $n$.  Thus for our purposes, $f(n)\preceq g(n)$ if $f(n) < dg(dn)$ for some positive constant $d$. Two functions $f, g$ are called {\em equivalent}, denoted $f\approx g$, if $f \preceq g$ and $g \preceq f$.\\

A function $f: \textbf{N} \to \textbf{N}$ is {\em superadditive} if for all natural numbers $m, n$ the inequality $f(m) + f(n) \leq f(m + n)$ holds.\\

In this paper we will work with objects (both group presentations and union machines) whose definitions depend on countably infinite sets of symbols.  We will require notions of being decidable and being c.e. for countable sets of symbols and for languages over such sets of symbols.  In order to define these notions, we adopt the following convention.  If $X$ is a countably infinite set of symbols, we require that the elements of $X$ to be of the form $x_{\textbf{y}}$, where the letter $x$ is an element of some finite alphabet and the index $\textbf{y}$ is a word over another finite alphabet.  From this perspective, both $X$ and any language $L$ over $X$ can be regarded as languages over a single finite alphabet.  Thus, if (\ref{pr3}) is a presentation with both $X$ and $R$ countably infinite, then the definitions of decidable sets and c.e. sets can be applied to both $X$ and $R$.  If a presentation (\ref{pr3}) has a finite or countably infinite generating set $X$, then we say that $P$ is decidable if both $X$ and $R$ are decidable.   Also, $P$ is c.e. if both $X$ and $R$ are c.e..  \\

For a group presentation $P=\langle X \| R\rangle$ where $X$ is finite or countably infinite, we say that the word problem is solvable for $P$ if the language of words in $F(X)$ that are trivial in $P$ is decidable.   Consider the language of pairs $(w,n)$ where $w$ is a trivial word in $F(X)$, $n$ is a positive integer written in unary, and there exists a $P$ diagram with boundary label $w$ and area not exceeding $n$.  If this language is decidable then we say that the {\em bounded word problem} for $P$ is solvable.  We say that the Dehn function $f$ of $P$ is computable if the language of pairs $(n,f(n))$, where $n$ and $f(n)$ are written in binary, is decidable.\\

We now formally define Turing machines and union machines.  In this paper all Turing machines are assumed to be nondeterministic.  A $k$-tape Turing machine has $k$ tapes and $k$ heads.  One can view it as a six-tuple

$$M=\langle A,\Gamma,Q,\Theta,\vec{s},\vec{h}\rangle,$$

\noindent where $A$ is the input alphabet and $\Gamma$ is the tape alphabet ($A\subseteq \Gamma$).  Each head has its own finite set of (disjoint) states, $Q_i$.  The set of states of the machine $M$ is $Q=Q_1\times...\times Q_k$.  An element of $Q$ is denoted $\vec{q}=(q_1,...,q_k)$ where $q_i\in Q_i$.  There is a state $\vec{s}\in Q$ called the start state, and a state $\vec{h}\in Q$ is called the accept state. \\

A {\em configuration} $\textbf{c}_i$ of the $i$th tape of $M$ is a word $\textbf{c}_i=\alpha_i \textbf{u}_i q_i \textbf{v}_i \omega_i$ where $q_i\in Q_i$ is the current state of the $i$th head, $\textbf{u}\in \Gamma^*$ is the word to the left of the head, and $\textbf{v}\in \Gamma^*$ is the word to the right of the head.  The letters $\alpha_i$ and $\omega_i$ are special letters of $\Gamma$ called the $i$th left and right end marker, respectively.  These letters may only appear at the left and right end of the $i$th tape; never anywhere else.  If the $i$th tape of a Turing machine is in configuration $\textbf{c}_i$, we say that $\textbf{u}_i\textbf{v}_i$ is the word written in the $i$th tape.  If $\textbf{u}_i,\textbf{v}_i=\varepsilon$ in $\textbf{c}_i$, we say that the $i$th tape is empty.\\

A {\em configuration} $\textbf{c}$ of $M$ is a $k$-tuple

$$\textbf{c}=(\textbf{c}_1,\textbf{c}_2,...,\textbf{c}_k),$$

\noindent where $\textbf{c}_i$ is a configuration of the $i$th tape.  The length $|\textbf{c}|$ of the configuration is the sum of the lengths of the words $\textbf{c}_i$.  The state of the configuration $\textbf{c}$ is the tuple $\vec{q}=(q_1,...,q_k)$, where $q_i$ is the $Q_i$ letter appearing in $\textbf{c}_i$.\\

If a tape letter $x$ is adjacent to a head letter $q_i$ in a configuration $\textbf{c}$, then we say that the head letter $q_i$ observes $x$ in $\textbf{c}$.\\

An {\em input configuration} is a configuration in which the word written in the first tape is in $A^*$, all other tapes are empty, the head of each tape observes its right marker $\omega_i$, and the state of the configuration is $\vec{s}$.  \\

The first tape of a Turing machine is called the {\em input tape}.  The input tape is a read only tape; that is, the word written on the input tape does not change during the course of a computation.  An {\em accept configuration} is any configuration for which the state is $\vec{h}$, the non-input tapes are empty, and the head of the first tape observes the right end letter $\omega_1$.\\

The commands of $M$ provide a way to pass from one configuration to another.  A command $\tau\in \Theta$ may be applied to a configuration $\textbf{c}$ of $M$, depending on the state of $\textbf{c}$ and the letters observed by the heads in $\textbf{c}$.  In a one-tape machine every command $\tau$ is of the following form:

$$uqv \to u'q'v',$$

\noindent where $u,v,u'v'$ are either letters of $\Gamma$ or the empty word $\varepsilon$.  The command $\tau$ can only be executed from a 1-tape configuration $\textbf{c}$ if $uqv$ is a subword of the single tape of \textbf{c}.  If the command $\tau$ is executed, then the machine replaces the subword $uqv$ with the subword $u'q'v'$.  In any such command $\tau$, $u= \alpha_1$ if and only if $u'=\alpha_1$.  Also $v=\omega_1$ if and only if $v'=\omega_1$.\\

Formally, a single-tape command is a six-tuple of symbols from the finite alphabets $Q$ and $\Gamma$, and  $\{\varepsilon\}$.  For example, the above command $\tau$ is the tuple $(u,q,v,u',q',v')$.\\

For a general $k$-tape machine, a command is a $k$-tuple of single-tape commands

$$\tau= (\tau_1, ..., \,\, \tau_k ), $$

\noindent where $\tau_i$ is $u_iq_iv_i \to u'_iq_i'v'_i$.  Formally then, such a command is a $6k$-tuple of symbols.  In order to execute the command $\tau$ from a $k$-tape configuration $\textbf{c}$, $u_iq_iv_i$ must be a subword of the configuration of the $i$th tape of $\textbf{c}$.  If this is the case, then the machine may execute the command $\tau$, replacing each $u_iq_iv_i$ by $u'_iq_i'v'_i$.  In any such command $\tau$, $u_i= \alpha_i$ if and only if $u'_i=\alpha_i$.  Also $v_i=\omega_i$ if and only if $v'_i=\omega_i$.\\

In later sections, we will want to represent each command of a Turing machine as a symbol.  The formal symbol for a command will be $\tau_{\vec{x}}$, where $\vec{x}$ is the corresponding $6k$-tuple.  The purpose of this notation is for a command symbol to contain a complete description of its corresponding command.  This will be quite important in the coming sections, as we will often have to effectively recover a command from its symbol. When referring to a command symbol informally, we will often omit the $\vec{x}$ index.\\

A {\em computation} of a $k$-tape Turing machine $M$ is a sequence of configurations $\textbf{c}^1,...,\textbf{c}^n$ such that for every $i=1,...,n-1$, the machine passes from $\textbf{c}^i$ to $\textbf{c}^{i+1}$ by applying one of the commands from $\Theta$.  A configuration $\textbf{c}$ is {\em acceptable} to $M$ if there exists at least one computation that starts with $\textbf{c}$ and ends with an accept configuration.  Such a computation is called an \textit{accepting} \textit{computation} for $\textbf{c}$.\\

An input word $\textbf{u}\in A^*$ is said to be {\em acceptable} if the input configuration for $\textbf{u}$ is an acceptable configuration.  The set of all acceptable input words over the alphabet $A$ is called the {\em language accepted by} $M$.\\

Let $C=(\textbf{c}^1,...,\textbf{c}^n)$ be a computation of a machine $M$ such that the configuraion $\textbf{c}^{j+1}$ is obtained from $\textbf{c}^j$ by the command $\tau_j\in \Theta$.  Then we call the word $\tau_1...\tau_{n-1}$ the {\em history} of the computation.  The number $(n-1)$ will be called the {\em time} or {\em length} of the computation.  The sum $\Sigma_{j=1}^n|\textbf{c}^j|$ will be called the {\em area} of $C$.\\

With every Turing machine we associate five functions: the {\em time function} $T(n)$, the {\em space function} $S(n)$, the {\em generalized time function} $T'(n)$, the {\em generalized space function} $S'(n)$, and the {\em area function} $A(n)$.  These functions will be called the {\em complexity functions} of the machine.  They are defined as follows.\\

We define $T(n)$ to be the minimal number such that every acceptable input configuration $\textbf{c}$ with $|\textbf{c}| \leq n$ is accepted by a computation of length at most $T(n)$.  The number $S(n)$ is the minimal number such that every acceptable input configuration $\textbf{c}$ with $|\textbf{c}|\leq n$ is accepted by a computation which contains only configurations of length $\leq S(n)$.  We define $T'(n)$ as the minimal number such that every acceptable configuration $\textbf{c}$ with $|\textbf{c}| \leq n$ is accepted by a computation of length at most $T'(n)$.  The number $S'(n)$ is the minimal number such that every acceptable configuration $\textbf{c}$ with $|\textbf{c}|\leq n$ is accepted by a computation which contains only configurations of length $\leq S'(n)$.  It is clear that $T(n)\leq T'(n)$ and $S(n)\leq S'(n)$ and it is easy to give examples where these inequalities are strict.  The area function $A(n)$ is defined as the minimal number such that for every acceptable configuration $\textbf{c}$ with $|\textbf{c}|\leq n$ there exists at least one accepting computation with area at most $A(n)$.\\

\begin{mydef}

Suppose $\{ M_i | i\in \textbf{N} \}$ is a countable set of Turing machines. Let $M_i=\langle A,\Gamma_i,Q_i,\Theta_i,\hat{s},\hat{h}\rangle$.   Note that each $M_i$ has identical $A$, $\vec{s}$, and $\vec{h}$.  We define $\bigcup_{i=1}^{\infty} M_i$ as follows:

$$\bigcup_{i=1}^{\infty} M_i = \langle A,\cup_{i=1}^{\infty} \Gamma_i, \cup_{i=1}^{\infty} Q_i, \cup_{i=1}^{\infty} \Theta_i,\vec{s},\vec{h}\rangle.$$

\end{mydef}

We call $\bigcup_{i=1}^{\infty} M_i$ a {\em union machine}.  All the terms and notation defined in this section concerning Turing machines have identical interpretations for union machines.  We will denote a union machine $\bigcup_{i=1}^{\infty} M_i$ as $M_{\infty}$.\\

A union machine $M_{\infty}$ is decidable if the sets $\bigcup_{i=1}^{\infty} \Gamma_i,\, \bigcup_{i=1}^{\infty} Q_i,\, \bigcup_{i=1}^{\infty} \Theta_i$ are decidable.  A union machine $M_{\infty}$ is c.e. if the sets $\bigcup_{i=1}^{\infty} \Gamma_i,\, \bigcup_{i=1}^{\infty} \, Q_i, \, \bigcup_{i=1}^{\infty} \Theta_i$ are c.e..\\

\section{Construction of Presentations Using Theorem \ref{main}}

It is trivial to show that the presentation $\langle a\| a\rangle$ is an example of case $\textbf{1}$.  In the following lemma we prove that examples of cases $\textbf{2}$ and $\textbf{5}$ do not exist.\\

\begin{lemma}\label{no2no5}

There are no finitely generated decidable presentations that satisfy case $\textbf{2}$ or case $\textbf{5}$.

\end{lemma}

\begin{proof}

For case \textbf{2}, we observe that if a finitely generated decidable presentation $P$ has solvable word problem and solvable bounded word problem, then we can compute the value of the Dehn function of $P$ on input $n>0$ as follows.  We first use the solvability of the word problem to effectively find all of the finitely many trivial words in $P$ of length $\leq n$.  Then, for each such trivial word $w$, we use the solvability of the bounded word problem to find the area of the minimal area $P$ diagram with boundary label $w$.  The largest such area will be the value of the Dehn function of $P$ on input $n$.  Therefore no finitely generated decidable presentation satisfies case $\textbf{2}$.\\

As for case \textbf{5}, if a finitely generated decidable presentation $P$ has solvable bounded word problem and computable Dehn function $f$, then the word problem for $P$ can be solved as follows.  To determine whether a word $w$ is trivial in $P$, first compute $f(|w|)$.  Then solve the bounded word problem for the input $(w,f(|w|))$.  If this input is accepted, then $w$ is trivial in $P$.  Otherwise, by the definition of the Dehn function, $w$ is not trivial in $P$.\\

\end{proof}

We will use Theorem \ref{main} to create examples of finitely generated minimal decidable group presentations satisfying the cases $\textbf{3},\textbf{4},\textbf{6},\textbf{7}$, and $\textbf{8}$.  Let $\textbf{B}\subseteq \{0,1\}^*$ be the set of binary representations of natural numbers.  In this section, when we want to indicate a particular element $x\in\textbf{B}$, we will often simply refer to the natural number for which $x$ is the binary representation.  For example, if we refer to the set of elements in $\textbf{B}$ that are $<1000$, we mean to indicate the set elements in $\textbf{B}$ that are binary representations of natural numbers $<1000$.\\

Let $\textbf{K}\subset \textbf{B}$ be a language over alphabet $\{0,1\}^*$ that is c.e. but not decidable.  It is possible to program a union machine $M_{\infty} ^{\textbf{K}}$ that can query the membership problem of $\textbf{K}$.  Such a union machine behaves similarly to an oracle Turing machine \cite{Soare}.\\

The machine $M_{\infty} ^{\textbf{K}}$ will query the membership problem of $\textbf{K}$ using its $k$th (and final) tape.  The $k$th set of state letters of $M_{\infty} ^{\textbf{K}}$ is $Q_k=\{q_k^{\textbf{y}}|\textbf{y}\in \{0,1\}^*\}$.  The other sets of state letters $Q_1,...,Q_{k-1}$ are required to be finite.  The machine never writes any letters in the $k$th tape.  Instead, it will use the upper index of the $k$th state letter as its $k$th tape.  We explain this formally below.\\

The commands of $M_{\infty} ^{\textbf{K}}$ come in two types: query commands and non-query commands.  We require that the $k$th component of a non-query command has one of the three following forms:
\begin{itemize}

\item
$\alpha_kq_k^{\textbf{y}}\omega_k\to \alpha_kq_k^{\textbf{y}a}\omega_k$,\\

\item
$\alpha_kq_k^{\textbf{y}a}\omega_k\to \alpha_k q_k^{\textbf{y}}\omega_k$,\\

\item
$\alpha_kq_k^{\textbf{y}}\omega_k\to \alpha_k q_k^{\textbf{y}}\omega_k$,\\

\end{itemize}

\noindent where $\textbf{y}\in \{0,1\}^*$ and $a\in \{0,1\}$.  We also require that the set of non-query commands be decidable.  Note that this is equivalent to requiring that the algorithm consisting of the non-query commands of $M_{\infty}^{\textbf{K}}$ could be performed by a standard non-deterministic Turing machine.  Thus, when we provide specific examples of such union machines, it will be sufficient to informally describe the algorithm run by the non-query commands of $M_{\infty}^{\textbf{K}}$.\\

The query commands of $M_{\infty} ^{\textbf{K}}$ are used to ask whether the upper index of the $k$th state letter of a given configuration is in $\textbf{K}$.  For each $\textbf{y}\in \textbf{K}$, there is a query command of the form

$$\tau_\textbf{y}= (q_1\to q'_1,...,q_{k-1}\to q'_{k-1}, q_k^\textbf{y} \to q_k^\textbf{y}).$$\\

The state $(q'_1,...,q'_{k-1},q_k^\textbf{y})$ is a ``yes" state, indicating that the element $\textbf{y}$ is a member of $\textbf{K}$.  If $\textbf{y}\notin \textbf{K}$, then no query command $\tau_\textbf{y}$ exists in $M_{\infty} ^{\textbf{K}}$.  In this way the machine $M_{\infty} ^{\textbf{K}}$ is different from an oracle Turing machine: while oracle machines may receive negative answers to queries, there is no way for $M_{\infty} ^{\textbf{K}}$ to receive a negative answer to a query. \\

The machine $M_{\infty} ^{\textbf{K}}$ is c.e..  This is because the input and work alphabets of $M_{\infty}^{\textbf{K}}$ are finite, the set of states of $M_{\infty} ^{\textbf{K}}$ is decidable, the set of non-query commands of $M_{\infty} ^{\textbf{K}}$ is decidable, and (since $\textbf{K}$ is c.e.) the set of query commands is c.e..\\

We can now use Theorem \ref{main} to construct group presentations with desired properties from union machines of the form $M_{\infty}^{\textbf{K}}$.  We will first construct a presentation to satisfy case \textbf{7}.  For the following construction, we assume that $\textbf{K}$ contains every even natural number.\\

We begin by describing a union machine $M_{\infty}^{\textbf{K}}$ with input alphabet $A$.  If $M_{\infty}^{\textbf{K}}$ is given $\textbf{u}\in A^*$ as an input word then $M_{\infty}^{\textbf{K}}$ writes the binary representation of $|\textbf{u}|$ in the upper index of the $k$th state letter.  If $|\textbf{u}|\in \textbf{K}$, there will be a query command in $M_{\infty}^{\textbf{K}}$ that can then be executed.  If this query command is executed then $M_{\infty}^{\textbf{K}}$ calculates $|\textbf{u}|^2$, runs for $|\textbf{u}|^2$ additional steps and accepts.    If there is no query command that can be applied (i.e. if $|\textbf{u}|$ is not in $\textbf{K}$), then there is no way for the machine to reach the accept state.  Thus the language accepted by $M_{\infty}^{\textbf{K}}$ is the set of input words whose lengths are in $\textbf{K}$.\\

Note that the length of an accepting computation of $M_{\infty}^{\textbf{K}}$ that begins with an input configuration for $\textbf{u}\in A^*$ is $O(|\textbf{u}|^2)$. Therefore the time function $T$ of $M_{\infty}^{\textbf{K}}$ is equivalent to $x^2$ and, by Theorem \ref{main}, the Dehn function $f$ of $P'_1(M_{\infty}^{\textbf{K}})$ is equivalent to $x^8$.  Since $\textbf{K}$ is undecidable, the language $L$ accepted by $M_{\infty}^{\textbf{K}}$ is not decidable. Therefore, by Theorem \ref{main} part 1, the word problem for $P'_1(M_{\infty}^{\textbf{K}})$ is not solvable.\\

Theorem \ref{main} tells us the Dehn function $f$ of $P'_1(M^\textbf{K}_{\infty})$ up to equivalence, but this is not sufficient to conclude that $f$ is computable.  For that, we need to prove the following lemma.  Consider the Baulmslag-Solitar presentation $H=\langle s,t\| sts^{-2}t^{-1}\rangle$, where $s,t$ are not among the generators of $P'_1(M^\textbf{K}_{\infty})$.

\begin{mydef}

Let $P$ be an arbitrary presentation.  For a word $w$ trivial in $P$, we define $L_P(w)$ to be the area of the minimal area $P$ diagram for $w$.

\end{mydef}

\begin{lemma}\label{forty}

Suppose $H$ is the Baulmslag-Solitar presentation defined above and $P$ is a finitely generated presentation whose Dehn function is equivalent to a superadditive polynomial function.  If $J$ is the direct product $H \times P$, then the Dehn function of $J$ is computable.

\end{lemma}

\begin{proof}

Let $f_J$ denote the Dehn function of $J$, $f_P$ denote the Dehn function of $P$, and $f_H$ denote the Dehn function of $H$.  It is well known that the Dehn function $f_H$ of the Baulmslag-Solitair presentation $H$ is equivalent to the exponential function $2^x$.\\

Since $f_H$ is equivalent to $2^x$, there is a constant $b$ such that for all $x\in\textbf{N}$, $2^{x}\leq bf_H(bx)$.  Therefore $\frac{1}{b}2^{\lfloor\frac{x}{b}\rfloor} \leq f_H(x)$ for all $x\in\textbf{N}$.  It follows that for all $n,m\in \textbf{N}$,

$$f_H(m+n) - f_H(n) \geq \frac{1}{b}(2^{\lfloor \frac{m+n}{b} \rfloor}-2^{\lfloor \frac{n}{b}\rfloor}) \geq \frac{1}{b}(2^{\lfloor\frac{m}{b}\rfloor+\lfloor\frac{n}{b}\rfloor}-2^{\lfloor\frac{n}{b}\rfloor}) =\frac{1}{b}2^{\lfloor\frac{n}{b}\rfloor}(2^{\lfloor\frac{m}{b}\rfloor}-1).$$

Therefore, if $m$ is sufficiently large (say $m>d$), then $f_H(n+m) - f_H(n) \geq (n+m)^2 + f_P(m)$ for all $n\in\textbf{N}$. \\

Let $w$ be an arbitrary reduced trivial word in $J$.  Let $|w|_H$ denote the number of $H$ letters in $w$, and $|w|_{P}$ denote the number of $P$ letters in $w$.  Then there is a diagram for $w$ with area at most $(|w|_P +|w|_H)^2+f_P(|w|_P)+f_h(|w|_H)$, where the $(|w|_P +|w|_H)^2$ term represents the number of relators required to transform $w$ into $w_1w_2$, where $w_1$ is a word in the generators of $H$ and $w_2$ is a word in the generators of $P$.  By the above inequalities, if $|w|_P\geq d$, then $f_H(|w|_P +|w|_H)\geq (|w|_P +|w|_H)^2+f_P(|w|_P)+f_h(|w|_H)$.\\

Thus if $w$ is trivial in $J$ and $|w|_P\geq d$, then $L_J(w)\leq f_H(|w|)$.  This means that for any word $w$ trivial in $J$ with $|w|_P\geq d$, there exists a word $w'$ in $H$ that is trivial in $J$ with $|w'|=|w|$ and $L_J(w')\geq L_J(w)$.  Informally, this means that if $w$ is trivial in $J$ and $|w|_P\geq d$, then $w$ need not be considered when computing $f_J$.\\

Therefore, if $w$ is a trivial word in $J$ with $L_J(w) = f_J(|w|)$, then $|w|_{P}<d$.  Since $P$ is finitely generated, the set $\Omega$ of trivial words in $P$ of length $\leq d$ is finite.  We collect the finite set of minimal area $P$ diagrams with boundary labels in $\Omega$, and then consider the finite set $R_1$ of $P$ relators that label 2-cells in these diagrams.  The label of a 2-cell that appears in a minimal area $J$ diagram for a word $w$ with $|w|_{P}<d$ is either the Baumslag-solitar relator $sts^{-2}t^{-1}$, a commuting relator $zyz^{-1}y^{-1}$ where $z\in H$ and $y\in J$, or an element of $R_1$.  Let $\tilde{P}$ be a presentation with the same generating set as $P$ and with relator set $R_1$.  If $w$ is a trivial word in $J$ with $L_J(w) = f_J(|w|)$ then the minimal area $J$ diagram for $w$ is identical to the minimal area $H\times \tilde{P}$ diagram for $w$.  Therefore $f_J$ is equal to the Dehn function of $H\times \tilde{P}$.  Since $H\times \tilde{P}$ is a finite presentation, we conclude that $f_J$ is computable.\\

\end{proof}

By Lemma \ref{forty}, the Dehn function of $H\times P'_1(M_{\infty}^{\textbf{K}})$ is computable.  Since the word problem for $H \times P(M_{\infty}^{\textbf{K}})$ is not solvable, and since no finitely generated decidable group presentation exists satisfying case $\textbf{5}$ (by Lemma \ref{no2no5}), the bounded word problem for $H \times P(M_{\infty}^{\textbf{K}})$ is unsolvable.  By Theorem \ref{main}, the fact that $H$ only contains one relator, and the obvious fact that a direct product of two minimal presentations is a minimal presentation, $H\times P'_1(M_{\infty}^{\textbf{K}})$ is minimal.\\

Next we consider case \textbf{3}.  For the following construction, we will assume that $\textbf{K}$, in addition to being c.e. and undecidable, contains no even numbers. \\

We create a union machine $M_{\infty}^{\textbf{K}}$ such that when $M_{\infty}^\textbf{K}$ is given an input word $\textbf{u}$, the machine $M^\textbf{K}_{\infty}$ writes the binary representation of $|\textbf{u}|$ in the upper index of the $k$th state letter and attempts to execute a query command.  If a query command is executed (which can only happen if $|\textbf{u}|\in \textbf{K}$) then $M^\textbf{K}_{\infty}$ accepts immediately.  Otherwise, $M^\textbf{K}_{\infty}$ calculates $|\textbf{u}|^{10}$, runs for $|\textbf{u}|^{10}$ additional steps, and then accepts.  Note that the time function of $M_{\infty}^\textbf{K}$ is equivalent to $x^{10}$, which is superadditive.\\

By Theorem \ref{main} part 2, since $M^{\textbf{K}}_{\infty}$ accepts every input, the word problem for $P'_1(M^\textbf{K}_{\infty})$ is solvable.  If the bounded word problem for $P'_1(M^\textbf{K}_{\infty})$ were solvable then it would be possible to decide $\textbf{K}$ as follows.  To decide if $n\in\textbf{K}$, pick an input word $\textbf{u}$ with $|\textbf{u}|=n$.  Consider the word $h(\mathcal{K}(\textbf{u}))$.  Since every input word is accepted by $M^\textbf{K}_{\infty}$, the word $h(\mathcal{K}(\textbf{u}))$ is trivial in $P'_1(M^\textbf{K})$.  If we begin solving the bounded word problem on inputs $(h(\mathcal{K}(\textbf{u})),1),(h(\mathcal{K}(\textbf{u})),2),...$, then we can find the area of the minimal area diagram with boundary label $h(\mathcal{K}(\textbf{u}))$.  If $n\in \textbf{K}$, then this area is $O((|\textbf{u}|^{10})^4)$ by Theorem \ref{main} part 4.  If $n\notin \textbf{K}$, then by Theorem \ref{main} part 4, this area will be $O(|\textbf{u}|^4)$.  Since $\textbf{K}$ is undecidable, we conclude that the bounded word problem is not solvable for $P'_1(M_{\infty}^{\textbf{K}})$.\\

Now we consider the presentation $H\times P'_1(M_{\infty}^{\textbf{K}})$, which (by Lemma \ref{forty}) will have computable Dehn function.   Since the bounded word problem is unsolvable for $P'_1(M_{\infty}^{\textbf{K}})$, the bounded word problem is also unsolvable for $H \times P'_1(M_{\infty}^{\textbf{K}})$.  Since the word problem is solvable for $H$ and $P'_1(M_{\infty}^{\textbf{K}})$, the word problem is also solvable for $H \times P'_1(M_{\infty}^{\textbf{K}})$.  By Theorem 3, the fact that $H$ only contains one relator, and the fact that a direct product of two minimal presentations is a minimal presentation, $H \times P'_1(M_{\infty}^{\textbf{K}})$  is minimal.\\

To provide an example for case $\textbf{4}$ we must first construct a function $f$ such that $f^4$ is superadditive and $f^4$ is not equivalent to any computable function.  To construct such an $f$, we first partition $\textbf{N}$ into disjoint subsets $I_n$, where $I_n$ is the set of integers contained in the interval $(10^{(n-1)^2},10^{n^2}]$.  The function $f$ will be constructed such that the values that $f$ takes on $I_n$ will depend on whether or not $n\in K$.\\

We now define the function $f$.  For $x\in I_n$, if $n\in\textbf{K}$ then $f(x)=10^{n^3}x^2$.  If $n \notin \textbf{K}$ then $f(x)=10^{n^3+2n^2}x^2$.  We claim that $f$ is superadditive.  To prove this claim, suppose that $x_1,x_2,x_3\in \textbf{N}$ and $x_1+x_2=x_3$.  Suppose also that $x_3\in I_n$.  Note that $(n-1)^3+2(n-1)^2< n^3$, which implies that for all $x<x_3$, $f(x)\leq 10^{n^3}x^2$.  Therefore,

$$f(x_1)+f(x_2) \leq 10^{n^3}x_1^2 + 10^{n^3}x_2^2 \leq 10^{n^3}(x_1+x_2)^2 = 10^{n^3}(x_3)^2 \leq f(x_3).$$

We now prove that $f^4$ is not equivalent to any computable function.  For any $x\in I_n$, if $n\in \textbf{K}$, then $f(x)\leq 10^{n^3}(10^{n^2})^2=10^{n^3+2n^2}$.  Alternately, if $n\notin \textbf{K}$ then $f(x)> 10^{n^3+2n^2}(10^{(n-1)^2})^2=10^{n^3+4n^2-4n+2}$. Suppose toward a contradiction that $f^4$ is equivalent to a computable function $f_c$.  Then there exists a constant $b_c$ such that for all $x\in \textbf{N}$,  $f_c(x)\leq b_cf^4(b_cx)$ and $f^4(x)\leq b_cf_c(b_cx)$.  Define $x_n:=\lfloor (10^{n^2})/({b_c^2})\rfloor$.  Note that for all sufficiently large $n$, $x_n\in I_n$.  Therefore, for sufficiently large $n$, the integers $x_n, b_cx_n, b_c^2 x_n$ are all contained in $I_n$.  Suppose that $n$ is sufficiently large that $x_n,b_cx_n, b_c^2 x_n\in I_n$.  If $n\notin\textbf{K}$, then $(10^{n^3+4n^2-4n+2})^4< f^4(x_n)\leq b_cf_c(b_cx_n)$.  Alternately, if $n\in\textbf{K}$, then $f_c(b_cx_n)\leq b_cf^4(b_c^2x_n)\leq b_c(10^{n^3+2n^2})^4$.\\

Thus if $n\notin\textbf{K}$, then $\frac{1}{b_c}(10^{n^3+4n^2-4n+2})^4\leq f_c(b_cx_n)$.  If $n\in\textbf{K}$, then $f_c(b_cx_n)\leq b_c(10^{n^3+2n^2})^4$. We note that if $n$ is sufficiently large then $\frac{1}{b_c}(10^{n^3+4n^2-4n+2})^4>b_c(10^{n^3+2n^2})^4$.  Thus for sufficiently large $n$, we can effectively decide whether $n\in \textbf{K}$ by evaluating $f_c(b_cx_n)$.  We conclude that $f^4$ is not equivalent to any computable function.\\

We now construct a union machine $M_{\infty}^{\textbf{K}}$ with time function equivalent to $f$ as follows.  The machine $M_{\infty}^{\textbf{K}}$ accepts every input.  When given input $\textbf{u}$, the machine $M_{\infty}^{\textbf{K}}$ writes $|\textbf{u}|$ in the upper index of the $k$th state letter and attempts to execute a query command.  If a query command is executed (which can only happen if $|\textbf{u}|\in \textbf{K}$) then $M_{\infty}^{\textbf{K}}$ calculates $10^{n^3}|\textbf{u}|^2$, runs for an additional $10^{n^3}|\textbf{u}|^2$ steps, and then accepts.  Otherwise, $M_{\infty}^{\textbf{K}}$ calculates $10^{n^3+2n^2}|\textbf{u}|^2$, runs for an additional $10^{n^3+2n^2}|\textbf{u}|^2$ steps, and then accepts.  Note that the time function $T$ of $M_{\infty}^{\textbf{K}}$ is equivalent to the function $f$ defined above.  Therefore $T^4$ is equivalent to $f^4$.\\

By Theorem \ref{main} part 2, since $M^{\textbf{K}}_{\infty}$ accepts every input, the word problem for $P'_1(M^\textbf{K}_{\infty})$ is solvable.  By Theorem \ref{main} part 3, the Dehn function of $P'_1(M^\textbf{K}_{\infty})$ is equivalent to $T^4$, and is therefore not computable.  If the bounded word problem for $P'_1(M^\textbf{K}_{\infty})$ were solvable then $P'_1(M^\textbf{K}_{\infty})$ would satisfy case \textbf{2}, which is impossible by Lemma \ref{no2no5}.  By Theorem \ref{main}, $P'_1(M^\textbf{K}_{\infty})$ is minimal.\\

To provide an example of case \textbf{6}, we can simply use the machinery from \cite{BRS}.  We let $M$ be a Turing machine that accepts the language $\textbf{K}$ (since $\textbf{K}$ is undecidable, the time function of $M$ is not bounded by any computable function).  For an input word $\textbf{u}$ of $M$, the word $\mathcal{K}(\textbf{u})$ is trivial in $P(M)$ if and only if $\textbf{u}\in \textbf{K}$.  Since $\textbf{K}$ is undecidable, the word problem for $P(M)$ is not solvable.  Since $P(M)$ is finite, it follows that the Dehn function of $P(M)$ is not computable.  Additionally, the bounded word problem is solvable for $P(M)$ because $P(M)$ is finite.  By Theorem \ref{smalldeal}, $P(M)$ is minimal.\\

For case \textbf{8}, if $P_1$ and $P_2$ are finitely generated decidable minimal presentations satisfying cases \textbf{6} and \textbf{7} respectively, then $P=P_1\times P_2$ satisfies case \textbf{8}.  Since the word problem for $P_1$ is not solvable and $P_1$ is a finite presentation, the Dehn function for $P_1$ is not bounded above by any computable function.  Therefore $P$ does not have computable Dehn function.  Since the word problem and bounded word problem are not solvable for $P_2$, they are also not solvable for $P$.  By Theorems \ref{smalldeal} and \ref{main} and the fact that the presentation of a direct product of two minimal presentations is a minimal presentation, $P=P_1\times P_2$ is minimal.\\

\section{Symmetrization of $M_{\infty}$}

The purpose of this section is to prove Lemma \ref{symsol}.  The proof of this lemma will rely on the construction in \cite{BRS} of a symmetric Turing machine $M'$ that simulates the Truing machine $M$.\\

For a Turing machine command $\tau$ of the form

$$(u_1q_1v_1\to u'_1q_1'v_1',...,u_kq_kv_k\to u'_kq_k'v'_k),$$\\

\noindent we write $\tau^{-1}$ to indicate the tuple

$$(u'_1q'_1v'_1\to u_1q_1v_1,...,u'_kq'_kv'_k\to u_kq_kv_k).$$\\

Note that $\tau^{-1}$ has the form of a command of a Turing machine.  These two commands are called $mutually$ $inverse$.  We say that a Turing machine is $symmetric$ if every $\tau \in \Theta$ has an inverse command $\tau^{-1} \in\Theta$.  The definition of a symmetric union machine is identical.  \\

We disallow Turing machine/union machine commands $\tau$ for which $\tau=\tau^{-1}$.  This can be done without loss of generality because when such a command $\tau$ is applied to a configuration $\textbf{c}$, the resulting configuration is $\textbf{c}$.  Thus in a symmetric machine the set of commands can be partitioned into {\em positive} and {\em negative} commands such that if $\tau$ is positive, then $\tau^{-1}$ is negative.\\

We will require a detailed description of how the symmetric Turing machine $M'$ is constructed from the standard Turing machine $M$ in \cite{BRS}.  An identical construction produces a symmetric union machine $M'_{\infty}$ from a standard union machine $M_{\infty}$.  Let $M_{\infty}$ be a $k$-tape union machine.  Recall that by our definitions of Turing machines and union machines, the first tape of $M_{\infty}$ is the {\em input} tape, which can only contain letters from the input alphabet.  Also, in an {\em input configuration} of $M_{\infty}$, an input word is written on the first tape, all other tapes are empty, and the head observes the right end marker of each tape.\\

The machine $M'_{\infty}$ has $k+1$ tapes.  As in $M_{\infty}$, the letters that can appear in the first tape are elements of $A$, while the letters that appear in tapes 2 through $k$ are letters of $\Gamma$.  The letters that are used in the $(k+1)$st tape of $M'_{\infty}$ are command symbols of $M_{\infty}$.  The input alphabet of $M'_{\infty}$ is identical to that of $M_{\infty}$.  We will define $M'_{\infty}$ by first describing the set of positive commands of $M'_{\infty}$.  The description of $M'_{\infty}$ will then be completed by including the inverses of the positive commands.\\

The $(k+1)$st set of state letters $Q_{k+1}$ contains three elements, $q(1),q(2),q(3)$.  The machine $M'_{\infty}$ is composed of three subroutines which we will call \textit{phases 1, 2, and 3}.  We call a configuration $\textbf{c}$ of $M'_{\infty}$ a {\em phase 1 configuration} if the $(k+1)$st state letter of $\textbf{c}$ is $q(1)$.  Similarly, phase 2 and phase 3 configurations have $(k+1)$st state letters $q(2)$ and $q(3)$, respectively.  Positive phase 1 commands can only be applied to phase 1 configurations, positive phase 2 commands can only be applied to phase 2 configurations, and positive phase 3 commands can only be applied to phase 3 configurations.\\

Input configurations of $M'_{\infty}$ are phase 1 configurations.  In a phase 1 configuration the $(k+1)$st head is in state $q(1)$, and the 1st through $k$th heads are in the start state of $M_{\infty}$.  For each command letter $\tau$ of $M_{\infty}$, there is a positive phase 1 command of $M'_{\infty}$ that writes the letter $\tau$ in the $(k+1)$st tape to the left of the $(k+1)$st head.  These commands do not change the state of $M'_{\infty}$.\\

After performing a phase 1 computation, the machine $M'_{\infty}$ will have a sequence of command symbols of $M_{\infty}$ written on the $(k+1)$st tape to the left of the $(k+1)$st head.  In order to proceed to the second phase, the machine checks if all tapes except tapes 1 and $(k+1)$ are empty and then changes the $(k+1)$st state to $q(2)$.  This is done by a single command of the form:

\begin{equation} \label{phase1-2}
(q_1\omega_1 \to q'_1 \omega_1,..., \alpha_i q_i\omega_i \to \alpha_i q'_i \omega_i,..., q(1)\omega_{k+1} \to q(2) \omega_{k+1})
\end{equation}\\

Note that in this command the $(k+1)$st state letter changes from $q(1)$ to $q(2)$.  In the second phase $M'_{\infty}$ attempts to use the first $k$ tapes to execute the sequence of commands written on tape $(k+1)$.  For every positive command $\tau$ of $M_{\infty}$ of the form

$$(u_1q_1v_1\to u'_1q_1'v_1',...,u_kq_kv_k\to u'_kq_k'v'_k),$$\\

\noindent we include the following positive command $\tau'$ in $M'_{\infty}$:

$$(u_1q_1v_1\to u'_1q_1'v_1',...,u_kq_kv_k\to u'_kq_k'v'_k, ~~\tau q(2)\to q(2)\tau).$$\\

The command $\tau'$ first checks if the command symbol $\tau$ is written to the left of the $(k+1)$st head (if $\tau$ is not written there, then the command $\tau'$ cannot be executed).  Then $\tau'$ executes the command $\tau$ on the first $k$ tapes of $M'_{\infty}$, and moves the $(k+1)$st head one letter to the left.\\

Suppose the $(k+1)$st head succeeds in moving all the way to the left end marker of the $(k+1)$st tape during phase 2.  Then the machine may pass to phase 3 provided tapes 1 through $k$ form an accept configuration of $M_{\infty}$ (recall that all tapes except the input tape are empty in an accept configuration of $M_{\infty}$).   In this case $M'_{\infty}$ may pass to phase 3 via the below command (in which $\vec{h}=(h_1,...,h_k)$ is the accept state of $M_{\infty}$).

$$(h_1\omega_1 \to h_1\omega_1, \alpha_2 h_2 \omega_2 \to \alpha_2 h_2 \omega_2,...,\alpha_k h_k \omega_k \to \alpha_k h_k \omega_k, \alpha_{k+1} q(2) \to \alpha_{k+1} q(3)).$$\\

In the third phase the machine erases tapes 1 and $(k+1)$ and enters the accept state of $M'_{\infty}$ once this erasing is complete.  This accept state is  $(h'_1,...,h'_k,q(3))$, where each $h'_i$ is a new state letter that is not a state letter of $M_{\infty}$.  Note that in phase 3 of $M'_{\infty}$ we make an exception to the rule that the input tape is read only.  We now include the inverses of all commands described above.\\

To complete the construction of $M'_{\infty}$, for each $i=1,...,k$, the $i$th tape of $M$ is divided into two tapes. These new tapes are numbered $2i-1$ and $(2i)$.  The new tapes $2i-1$ and $(2i)$ simulate the portions of the old $i$th tape that lay to the left and right of the head, respectively.\\

If a configuration of the old tape $i$ was $\alpha_1 \textbf{u}q_i\textbf{v}\omega_i$ then the corresponding configurations of the new tapes $2i-1$ and $2i$ will be, respectively:

$$\alpha_i\textbf{u}q_i \omega_i \,\,\,\,\,\,\, \text{and} \,\,\,\,\,\,\,\, \alpha_{(i+1/2)} \bar{\textbf{v}} q_{(i+1/2)} \omega_{(i+1/2)},$$\\

\noindent where $\bar{\textbf{v}}$ is the word $\textbf{v}$ rewritten from right to left.\\

The set of commands is then adjusted so that each old command is replaced by $2k$ new commands that execute the old command one tape at a time.  These adjustments (formally described on page 399 of \cite{BRS}) do not affect the language accepted by the machine.  The complexity functions are changed by only a constant factor (since each old command has been turned into $2(k+1)$ new commands).  As a result of these adjustments, every command of $M'_{\infty}$ is of one of the two following forms. \\

\begin{equation} \label{type1}
(q_1\omega_1 \to q'_1 \omega_1,..., aq_i\omega_i \to q'_i \omega_i,... ,q_{2(k+1)}\omega_{2(k+1)} \to q'_{2(k+1)} \omega_{2(k+1)}).
\end{equation}

\begin{equation} \label{type2}
(q_1\omega_1 \to q'_1 \omega_1,...,\alpha_i q_i\omega_i \to \alpha_i q'_i \omega_i, ..., q_{2(k+1)}\omega_{2(k+1)} \to q'_{2(k+1)} \omega_{2(k+1)}).
\end{equation}\\

\begin{observation}\label{symce}

If $M_{\infty}$ is c.e. then $M'_{\infty}$ is c.e..

\end{observation}

This observation follows immediately from the above description of the construction of $M'_{\infty}$ from $M_{\infty}$.\\

The properties of $M'_{\infty}$ that are used in \cite{BRS} to prove Theorem \ref{smalldeal} are listed in \cite[Lemma 3.1]{BORS}.  In particular, $M_{\infty}$ and $M'_{\infty}$ accept the same language.\\

\begin{lemma} \label{symsol}

Suppose $M_{\infty}$ is a $k$-tape union machine accepting a language $L$.  Then $M'_{\infty}$ has solvable configuration problem if and only if $L$ is decidable.  \\

\end{lemma}

\begin{proof}

In order to determine whether a phase 1 configuration $\textbf{c}$ is acceptable, we first check if all tapes of $\textbf{c}$ except the left half of the input tape (i.e. tape 1) and the left half of the history tape (i.e. tape $(2(k+1)-1$) are empty.  If not, then $\textbf{c}$ is not an acceptable configuration because no computation starting with $\textbf{c}$ can ever exit phase 1.  If all tapes besides the first and the $(2(k+1)-1)$st are empty, then $\textbf{c}$ is an acceptable configuration of $M'_{\infty}$ if and only if the word $w$ written on the input tape of $\textbf{c}$ is an accepted input of $M_{\infty}$.  This is because there is a computation of $M'_{\infty}$ beginning with $\textbf{c}$ that simply erases the left half of the history tape of $\textbf{c}$, leaving the machine in an input configuration $\textbf{c}'$ for the input word $\textbf{u}$.  Since $M_{\infty}$ and $M'_{\infty}$ accept the same language, $M'_{\infty}$ accepts $\textbf{c}'$ (and therefore $\textbf{c}$) if and only if $w\in L$ accepted by $M_{\infty}$.  Thus, if $L$ is decidable, we can decide whether $\textbf{c}$ is acceptable. \\

A phase 3 configuration $\textbf{c}$ is acceptable if all tapes except the 1st (i.e. the left half of input tape) and $2(k+1)$st (i.e. the right half of the history tape) are empty, and the state of $\textbf{c}$ is $(h_1,\dots,h_k,q(3))$.  Alternately, $\textbf{c}$ is acceptable if $\textbf{c}$ has state $(h'_1,\dots,h'_k,q(3))$ and all tapes are empty.  These are the only acceptable phase 3 configurations.\\

If a phase 2 configuration $\textbf{c}$ is an acceptable configuration of $M_{\infty}$, then there must be a phase 2 computation that starts in $\textbf{c}$ and ends in either a phase 1 configuration or a phase 3 configuration.  A reduced phase 2 computation beginning with $\textbf{c}$ must execute a sequence of commands written in one of the two halves of the history tape in $\textbf{c}$.  Recall that each command symbol of $M_{\infty}$ written in these history tapes contains a complete description of the corresponding $M_{\infty}$ command.  Therefore we can effectively recover the finite set of $M'_{\infty}$ commands whose command symbols are written in the history tapes of $\textbf{c}$.  We can then effectively check if applying either of the sequences of commands written in the history tapes ends in either a phase 1 or a phase 3 configuration.  If not, then $\textbf{c}$ is not an acceptable configuration.  If a phase 1 or phase 3 configuration $\textbf{c}'$ can be reached from $\textbf{c}$, then $\textbf{c}$ is acceptable if and only if $\textbf{c}'$ is acceptable.  By the above two paragraphs, we can effectively decide whether $\textbf{c}'$ is acceptable.\\

\end{proof}

\section{Properties of $S(M'_{\infty})$}

An $S$-machine is a group presentation of an HNN-extension of a free group that satisfies some additional conditions.  Certain words in the generators of the base group of such an HNN-extension are thought of as configurations of the S-machine.  As in \cite{BRS}, we call these configurations {\em admissible words}.  The stable letters of the HNN-extension are thought of as the commands of the S-machine.  These stable letters act on the set of admissible words by conjugation.  We formalize this idea below.\\

A {\em hardware} of an $S$-machine is a free group $G=F(\hat{A}\cup \hat{Q})$, where $\hat{A}$ and $\hat{Q}$ are disjoint sets of positive generators.  The hardware will be the base group of the HNN-extension.  We will call $\hat{A}$ the set of tape letters, and $\hat{Q}$ the set of state letters.  The set $\hat{Q}$ is the union of $k$ disjoint sets: $\hat{Q}=\hat{Q}_1\cup...\cup \hat{Q}_k$.\\

A reduced word $w$ in the generators of $G$ is an {\em admissible word} of $G$ if it has the form $w= r_1 w_1 r_2 w_2...r_{k-1}w_{k-1}r_k$, where $r_i\in \hat{Q}_i$ and $w_i\in F(\hat{A})$.  If $i\leq j$, then a subword of an admissible word $w$ of the form $r_iw_i....r_{j}$ is called the $(i,j)$-{\em subword} of $w$.  Note that if $i=j$ then an $(i,j)$-subword consists of only a single $\hat{Q}_i$ letter. \\

An $S$-machine $\mathcal{N}$ is an HNN extension of a hardware $G$.  The set of stable letters of this HNN extension is $\hat{\Theta}=\{\rho_1,\dots,\rho_n,\dots\}$, or the set of {\em command letters} of $\mathcal{N}$.  In \cite{BRS}, it is required that $\hat{A},\hat{Q}_1,\dots, \hat{Q}_k$, and $\hat{\Theta}$ be finite.  We allow these sets to be countably infinite.\\

If $\mathcal{N}$ is an $S$-machine with hardware $G$, then $\mathcal{N}$ is of the form
$$\langle G,\rho_1,\dots,\rho_n \dots \| H_1,\dots,H_n \dots \rangle,$$

\noindent where the $H_i$'s are disjoint sets of relators.  The set $H_i$ corresponds to the stable letter $\rho_i$: every relator in $H_i$ has the form $\rho^{-1}_i x \rho_i=y$, where $x,y$ are words in the generators of $G$.  Also, the letter $\rho_i$ does not appear in the relators of $H_j$ if $j\neq i$.\\

There are two types of relators in each $H_i$: transition relators and auxiliary relators.  The auxiliary relators are $\{\rho_i a \rho_i^{-1}=a|a\in \hat{A}\}$.  The transition relators of $H_i$ have the form $\rho^{-1}_i \textbf{u} \rho_i=a\textbf{v}b$ where, for some $m\leq n\leq k$, the words $\textbf{u},\textbf{v}$ are both $(m,n)$-subwords of admissible words and $a,b\in \hat{A}^{\pm 1} \cup \{\varepsilon\}$.\\

For each $\hat{Q}_j$, there is exactly one transition relator in each $H_i$ in which a $\hat{Q}_j$ letter appears.  Furthermore, if $\rho^{-1}_i \textbf{u} \rho_i=a\textbf{v}b$ is the single $H_i$ relator in which a $\hat{Q}_j$ letter appears then there is exactly one $\hat{Q}_j$ letter in each of the words $\textbf{u}$ and $\textbf{v}$ (this follows from the fact that both $\textbf{u}$ and $\textbf{v}$ are $(m,n)$ subwords of admissible words).\\

We now prove that an $S$-machine is an HNN-extension of its hardware.

\begin{lemma}\label{hnn}

Let  $\mathcal{N}=\langle G,\rho_1,\dots,\rho_n\dots \| H_1,\dots,H_n \dots\rangle$ be an $S$-machine.  Then $\mathcal{N}$ is an HNN-extension of $G$.
\end{lemma}

\begin{proof}

Suppose $\textbf{r}_1,\textbf{r}_2,\dots,\textbf{r}_s\dots$ are the relators of $H_i$.  Then each $\textbf{r}_{\ell}$ is of the form $\rho^{-1}_i \textbf{x}_{\ell} \rho_i=\textbf{y}_{\ell}$ where $\textbf{x}_{\ell},\textbf{y}_{\ell}$ are words in the generators of $G$.  Note that for each $\ell=1,2,\dots$, the word $\textbf{x}_{\ell}$ is either a single letter of $\hat{A}$ (if $\textbf{r}_{\ell}$ is an auxiliary relator) or an $(m_{\ell},n_{\ell})$-subword of an admissible word (if $\textbf{r}_{\ell}$ is a transition relator).  Each such $(m_{\ell},n_{\ell})$-subword contains at least one $\hat{Q}$ letter, and no two distinct such $(m_{\ell},n_{\ell})$-subwords share a common $\hat{Q}$ letter.  Therefore the elements $\textbf{x}_1,\dots,\textbf{x}_s\dots$ are Nielsen reduced and they are free generators of a subgroup of $G$.\\

We will show that the elements $\textbf{y}_1,\dots,\textbf{y}_s\dots$ also freely generate a subgroup of $G$.  For all $a\in \hat{A}$ there exists an auxiliary relator $\textbf{r}_{\ell}\in H_i$ such that $\textbf{y}_{\ell}=a$.  If $\textbf{r}_{\ell}$ is a transition relator then $\textbf{y}_{\ell}$ is a word of the form $a\textbf{v}b$, where $\textbf{v}$ is an $(m_{\ell},n_{\ell})$-subword of an admissible word and $a,b\in \hat{A}^{\pm 1} \cup \{\varepsilon\}$.  If, for a given $\textbf{y}_{\ell}$, the letter $a$ (or $b$) is not $\varepsilon$ then we can perform a Neilson reduction on $\textbf{y}_1,\dots,\textbf{y}_s\dots$ to remove $a$ (or $b$) from $\textbf{y}_{\ell}$.  After we remove all such $a$ and $b$ letters from the words $\textbf{y}_1,\dots,\textbf{y}_s\dots$, the resulting set of words is Neilson reduced by the argument in the above paragraph.  Therefore the elements $\textbf{y}_1,\dots,\textbf{y}_s\dots$ freely generate a subgroup of $G$, and the map $\textbf{x}_{\ell}\mapsto \textbf{y}_{\ell}$ induces an isomorphism of subgroups of $G$.  We conclude that $\mathcal{N}$ is an HNN-extension of $G$.\\

\end{proof}

If $\mathcal{N}_1$ is an $S$-machine, we say that an $S$-machine $\mathcal{N}_2$ is a {\em submachine} of $\mathcal{N}_1$ if every generator of the hardware of $\mathcal{N}_2$ is a generator of the hardware of $\mathcal{N}_1$, every command letter of $\mathcal{N}_2$ is a command letter of $\mathcal{N}_1$, and every relator of $\mathcal{N}_2$ is a relator of $\mathcal{N}_1$.\\

If $W_1$ and $W_2$ are admissible words of $\mathcal{N}$ and the equation $W_1=\rho^{\pm1} W_{2} \rho^{\mp 1}$ holds in $\mathcal{N}$, then we say that the command letter $\rho^{\pm1}$ can be applied to $W_1$ and the command letter $\rho^{\mp1}$ can be applied to $W_2$.  A {\em computation} of an $S$-machine $\mathcal{N}$ is a sequence of admissible words $W_1,...,W_n$ such that for each $i\geq 2$, the equation $W_i=\rho^{\pm1} W_{i-1} \rho^{\mp 1}$ holds in $\mathcal{N}$ for some command letter $\rho$ of $\mathcal{N}$.\\

For an $S$-machine $\mathcal{N}$, we may designate a single admissible word $W_0$ as the ``accept configuration" of the machine.  We say that an admissible word $W$ is {\em acceptable} by $\mathcal{N}$ if there is a computation of $\mathcal{N}$ that begins with $W$ and ends with $W_0$.\\

We define the complexity functions of an $S$-machine the same way we defined them for Turing machines: simply replace the word ``configuration" with ``admissible word".\\

In \cite{BRS}, the authors construct an $S$-machine $S(M')$ to simulate the symmetrization $M'$ of an arbitrary Turing machine $M$.  This construction does not rely on the finiteness of $M'$.  The arguments given in \cite{BRS} actually prove that the exact same construction can be used to produce an $S$-machine $S(M'_{\infty})$ to simulate the symmetrization $M'_{\infty}$ of an arbitrary union machine $M_{\infty}$.  In this section, we will often cite lemmas of \cite{BRS} as though they were statements about $S(M'_{\infty})$ instead of $S(M')$.  When we do this, it should be understood that the proofs of those lemmas about $S(M')$ as they are stated in \cite{BRS} suffice to prove the corresponding lemmas about $S(M'_{\infty})$ as well. \\

The simulation of $M'_{\infty}$ by $S(M'_{\infty})$ relies on an injective map $\sigma$ from the set of configurations of $M'_{\infty}$ to the set of admissible words of $S(M'_{\infty})$.  The definition of $\sigma(\textbf{c})$ appears on page 400 of \cite{BRS}.  \\

\begin{lemma}\label{recover}

Given an admissible word $W$ of $S(M'_{\infty})$, it is possible to decide in linear time whether or not $W=\sigma(\textbf{c})$ for some configuration $\textbf{c}$ of $M'_{\infty}$.  Also, in the case that $W=\sigma(\textbf{c})$, it is possible to effectively recover $\textbf{c}$ from $W$ in linear time.

\end{lemma}

\begin{proof}

This lemma follows immediately from the definition of $\sigma(\textbf{c})$ given in \cite{BRS}.

\end{proof}

There are some additional facts about the $S(M'_{\infty})$ construction that we will require.  The machine $S(M'_{\infty})$ is a union of disjoint submachines $\mathcal{R}_{\tau}$, each of which corresponds to a command $\tau$ of $M'_{\infty}$.  Each $\mathcal{R}_{\tau}$ contains finitely many command letters.  The purpose of the submachine $\mathcal{R}_{\tau}$ is to allow $S(M'_{\infty})$ to pass from $\sigma(\textbf{c})$ to $\sigma(\textbf{c}')$ if and only if the $M'_{\infty}$ command $\tau$ takes $\textbf{c}$ to $\textbf{c}'$.  The command letters of $\mathcal{R}_{\tau}$ are all indexed by the $M'_{\infty}$ command symbol $\tau$.\\

If a command $\tau$ of $M'_{\infty}$ is of the form (\ref{type1}) then $\mathcal{R}_{\tau}$ is itself composed of several submachines which are denoted in \cite{BRS} by $S_4(\tau)$, $S_9(\tau)$, $R_4(\tau)$, $R_{4,9}(\tau)$, and $R_9(\tau)$.  If a command $\tau'$ of $M'_{\infty}$ is of the form (\ref{type2}), then $\mathcal{R}_{\tau'}$ is composed of a single submachine $P(\tau')$, which contains only a single command letter.  A concise description of the function of each of these machines can be found on pages 397-398 of \cite{BRS}. Their formal definitions are located on pages 374-396 of \cite{BRS}.  Note that the notation $\mathcal{R}_{\tau}$ does not appear in \cite{BRS}.  We use it here for convenience.\\

\begin{lemma} \label{comenu}

There is an algorithm that, when given as input an $M'_{\infty}$ command $\tau$ of type (\ref{type1}) or $\tau'$ of type (\ref{type2}), outputs the set of transition relators of $\mathcal{R}_{\tau}$ or $\mathcal{R}_{\tau'}$, respectively.

\end{lemma}

\begin{proof}

The definitions of $S_4(\tau)$, $S_9(\tau)$, $R_4(\tau)$, $R_{4,9}(\tau)$, $R_9(\tau)$, and $P(\tau')$ given in \cite{BRS} describe exactly how to effectively construct these transition relators from $\tau$ (or $\tau'$).  The ``algorithm" referred to in the lemma is simply the process of following these instructions.

\end{proof}

For any command $\tau$ or $\tau'$ in $M'_{\infty}$ (of type (\ref{type1}) or (\ref{type2}) respectively), the submachines $R_4(\tau)$, $R_9(\tau)$, and $P(\tau')$ each contain a single positive command letter.  In the case of each of these three submachines, we will use the same notation to denote both the machine and its single positive command letter.  For example if we say that $R_4(\tau)^{-1}$ can be applied to an admissible word, then we mean that the inverse of the single positive command letter contained in the submachine $R_4(\tau)$ can be applied to that admissible word.\\

The state letters of $S(M'_{\infty})$ are divided into two types: standard and non-standard.  Each non-standard state letter of $S(M'_{\infty})$ is indexed by a positive command symbol of $M'_{\infty}$.  The standard state letters are not indexed by any command symbols of $M'_{\infty}$.  A complete description of the state letters of $S(M'_{\infty})$ can be found on page 397 of \cite{BRS}.\\

\begin{lemma} \label{standard}

Suppose $W$ is an admissible word of $S(M'_{\infty})$ such that all state letters appearing in $W$ are standard.  Then the only command letters of $S(M'_{\infty})$ that may be applied to $W$ are $R_4(\tau)$, $R_9(\tau)^{-1}$, or $P(\tau')^{\pm 1}$ for $M'_{\infty}$ commands $\tau$ of type (\ref{type1}) or $\tau'$ of type (\ref{type2}).

\end{lemma}

\begin{proof}

This follows immediately from the description of the commands of $S(M'_{\infty})$ on pages 397-399 of \cite{BRS}.

\end{proof}

In \cite{BRS}, the authors call an admissible word $W$ {\em normal} if it fulfills certain properties (the definition is on page 400 of \cite{BRS}).  For our purposes, the details of this definition are not important.  It will suffice to note that it is stated on page 403 of \cite{BRS} that the commands of $S(M'_{\infty})$ take normal words to normal words, and that every admissible word $\sigma(\textbf{c})$ is a normal word. \\

\begin{lemma} \label{normal}

Let $W$ be an admissible word of $S(M'_{\infty})$.  Suppose that $W$ is positive and normal.  Suppose also that one of the command letters $R_4(\tau)$, $R_9(\tau)^{-1}$, $P(\tau')$ can be applied to $W$.  Then $W=\sigma(\textbf{c})$ for some configuration $\textbf{c}$ of $M'_{\infty}$.

\end{lemma}

\begin{proof}

This is Lemma 4.15\cite{BRS}.

\end{proof}

\begin{corollary}\label{cor1}

Let $W$ be an admissible word of $S(M'_{\infty})$.  Suppose that $W$ is positive and normal, and that a command letter $P(\tau')^{-1}$ can be applied to $W$.  Then $W=\sigma(\textbf{c})$ for some configuration $\textbf{c}$ of $M'_{\infty}$.

\end{corollary}

\begin{proof}

In \cite{BRS}, Lemma 4.15\cite{BRS} is stated without proof because it follows immediately from the definition of $S(M'_{\infty})$.  Similarly, this corollary (not stated in \cite{BRS}) follows immediately from the definition of $S(M'_{\infty})$ as well.

\end{proof}

\begin{lemma}\label{standard2}

Suppose $W$ is an admissible word of $S(M'_{\infty})$ and there is a computation $C$ of $S(M'_{\infty})$ that starts with $W_0$ and ends with $W$.  If every state letter appearing in $W$ is standard, then $W=\sigma(\textbf{c})$ for some configuration $\textbf{c}$ of $M'_{\infty}$.

\end{lemma}

\begin{proof}

As stated on page 403 of \cite{BRS}, any computation $C$ of $S(M'_{\infty})$ can be represented in the form

$$C=C_1...C_N,$$

\noindent where each $C_i$ is a non-empty computation of one of the sub-machines:

$$S_4(\tau), S_9(\tau), R_4(\tau), R_{4,9}(\tau), R_9(\tau),P(\tau')$$\\

\noindent where $\tau$ is an $M'_{\infty}$ command of the form (\ref{type1}) and $\tau'$ is an $M'_{\infty}$ command of the form (\ref{type2}).  No two consecutive computations $C_i$ and $C_{i+1}$ come from the same submachine.\\

Since $C$ ends in $W$, which is an admissible word with all state letters standard, it follows from Lemma \ref{standard} that the machine that executes $C_N$ must be either $P(\tau')^{\pm1}$, $R_4(\tau)$ or $R_9(\tau)^{-1}$ for some $\tau$ of type \ref{type1} or $\tau'$ of type \ref{type2}.  It is stated on page 406 of \cite{BRS} that the first word in the computation $C_N$ is positive.  By definition, each of the command letters $P(\tau_i)^{\pm1}$, $R_4(\tau_i)$, and $R_9(\tau_i)^{-1}$ take positive admissible words to positive admissible words.  We conclude that $W$ is positive.  Since $C$ began with $W_0$, all words in $C$ are normal.  The result now follows from Lemma \ref{normal} and Corollary \ref{cor1}.

\end{proof}

The following lemma is stated in the proof of Proposition 4.1\cite{BRS} on pages 408-409.

\begin{lemma}\label{durp}

If $W$ is an acceptable admissible word of $S(M'_{\infty})$, then there is a computation $C''_1$ of $S(M'_{\infty})$ that takes $W$ to $\sigma(\textbf{c})$ for some acceptable configuration $\textbf{c}$ of $M'_{\infty}$.  The computation $C''_1$ is composed of computations of $S_4(\tau)$, $S_9(\tau)$, $R_4(\tau)$, $R_{4,9}(\tau)$, $R_9(\tau)$ for some command $\tau$ of $M'_{\infty}$.  The length of $C''_1$ does not exceed $O(|W|^2)$.

\end{lemma}

The following lemma follows directly from the construction of $S(M'_{\infty})$ in \cite{BRS}.  Specifically, it follows from the fact that for all $k$-tape union machines, the lengths of the transition relators in the sub-machines $S_4(\tau)$, $S_9(\tau)$, $R_4(\tau)$, $R_{4,9}(\tau)$, $R_9(\tau)$, $P(\tau')$ are invariant under the choice of $\tau$ and $\tau'$.  This lemma can be verified from the summary of the rules of $S(M'_{\infty})$ given on pages 397-399 of \cite{BRS} and from the formal definitions in \cite{BRS} of each of the submachines mentioned in that summary.

\begin{lemma}\label{rellength}

There is a constant bound $b_k$ such that for any union machine $M_{\infty}$ the relators in the presentation $S(M'_{\infty})$ have length less than $b_k$.

\end{lemma}

\begin{lemma}\label{smachsol}

If the configuration problem is solvable for $M'_{\infty}$, then the configuration problem is solvable for $S(M'_{\infty})$.

\end{lemma}

\begin{proof}

Let $W$ be an accepted admissible word of $S(M'_{\infty})$.  If $W=\sigma(\textbf{c})$ for some configuration $\textbf{c}$ of $M'_{\infty}$, then by Lemma \ref{recover} we can recover $\textbf{c}$ from $W$.  By \cite[Propositoin 4.1]{BRS}, $\textbf{c}$ is accepted by $M'_{\infty}$ if and only if $\sigma(c)$ is accepted by $S(M'_{\infty})$.  Thus we can decide whether $W=\sigma(\textbf{c})$ is accepted by $S(M'_{\infty})$.\\

If $W$ is not equal to $\sigma(\textbf{c})$ for any configuration $c$ of $M'_{\infty}$, then by Lemma \ref{standard2} at least one state letter of $W$ is non-standard.  We choose a non-standard state letter of $W$ and look at its $\hat{\Theta}$ index $\tau$.  By Lemma \ref{comenu}, we can use $\tau$ to recover the finite set of transition relators of the sub-machine $\mathcal{R}_{\tau}$.  We use these to effectively construct all of the finitely many computations of the sub-machine $\mathcal{R}_{\tau}$ of length not exceeding $O(|W|^2)$ that begin with the admissible word $W$.  We then check if any of these computations end with an admissible word of the form $\sigma(\textbf{c})$.  If not, then by Lemma \ref{durp}, $W$ is not an accepted admissible word of $S(M'_{\infty})$.  If so, then we collect the finitely many configurations $\textbf{c}$ of $M'_{\infty}$ such that $\sigma(\textbf{c})$ is reachable from $W$ by applying such a computation.  The admissible word $W$ is accepted by $S(M'_{\infty})$ if and only if at least one of these configurations $\textbf{c}$ is an accepted configuration of $M'_{\infty}$.  Since the configuration problem is solvable for $M'_{\infty}$, the proof is complete.

\end{proof}

\begin{lemma}\label{smachce}

If $M'_{\infty}$ is c.e. then $S(M'_{\infty})$ is c.e.

\end{lemma}

\begin{proof}

By Lemma \ref{comenu}, the set of transition relators of $S(M'_{\infty})$ can be effectively enumerated from the set of commands of $M'_{\infty}$.  Every state and command letter of $S(M'_{\infty})$ appears in a transition relator of $S(M'_{\infty})$.  The set of tape letters of $S(M'_{\infty})$ consists of the alphabet letters of $M'_{\infty}$ plus finitely many additional letters, as described on page 396 of \cite{BRS}.  The set of auxiliary relators of $S(M'_{\infty})$ can be effectively constructed from the sets of command and alphabet letters of $S(M'_{\infty})$.  Thus if we can computably enumerate $M'_{\infty}$, we can computably enumerate $S(M'_{\infty})$.\\

\end{proof}

\section{Diagrams}

In this section we define the terms and notation that we will use to discuss van Kampen diagrams.  Let $P=\langle X \parallel R \rangle$ be a group presentation.\\

A {\em van Kampen diagram} (or often just a diagram) over the presentation $P$ is a planar, finite, connected and simply connected 2-complex $\Delta$.  The edges of $\Delta$ are oriented, and each oriented edge of $\Delta$ is labeled by an element of $X^{\pm 1}$.  The label of an oriented edge $e$ is denoted $\text{Lab}(e)$.  If $\text{Lab}(e)=x$, then $\text{Lab}(e^{-1})=x^{-1}$.  If $p=e_1\dots e_k$ is a path of edges in $\Delta$ then the label of $p$ is $\text{Lab}(p)=\text{Lab}(e_1)\dots \text{Lab}(e_k)$.  We denote the boundary path of $\Delta$ by $\partial\Delta$. For a 2-cell $\pi$ of $\Delta$, we denote the boundary path of $\pi$ by $\partial\pi$ (we will often write ``boundary" for boundary path).  We will always assume that boundary paths of diagrams and of 2-cells are oriented in the clockwise direction.  For every 2-cell $\pi$ of $\Delta$, $\text{Lab}(\partial\pi)$ is an element of the relator set $R$.  The {\em contour} of a cell or diagram is the union of its boundary and the inverse of its boundary.  An {\em undirected edge} of $\Delta$ is the union of an edge of $\Delta$ with its inverse.  An edge $e$ of $\Delta$ is a {\em boundary edge} of $\Delta$ if $e$ is contained in $\partial \Delta$.  If an edge $e$ of $\Delta$ is not a boundary edge then $e$ is an {\em interior edge}.  We call the initial and final vertices of an edge $e$ the {\em endpoints} of $e$. \\

A pair of 2-cells $\pi_1, \pi_2$ in a diagram $\Delta$ is called a {\em reducible pair} if there is an edge $e$ in $\Delta$ such that $\partial \pi_1 = ve$, $\partial \pi_2 = e^{-1}v'$ and $\text{Lab}(v)=\text{Lab}(v')^{-1}$.  A diagram is {\em reduced} if it contains no reducible pairs.\\

The length of a path $p$, denoted $|p|$, is the number of edges $p$ contains.  The {\em area} of a diagram $\Delta$ is the number of 2-cells $\Delta$ contains.  If we say that $\Delta$ is a {\em minimal area} diagram over some presentation $P$, we mean that $\Delta$ is the minimal area $P$ diagram with boundary label $\text{Lab}(\partial\Delta)$.\\

Let \textbf{S} be a subset of the generating set of a presentation $P$. An \textbf{S}-band $\mathcal{B}$ over $P$ is a sequence of 2-cells $\pi_1,...,\pi_n$ in a van Kampen diagram such that:

 \begin{itemize}

 \item
For $i=1,\dots,(n-1)$, the boundaries $\partial\pi_i$ and $(\partial\pi_{i+1})^{-1}$ share a common edge labeled by a letter from \textbf{S}.\\

\item
For $i=1,...,n$, the boundary $\partial\pi_i$ contains exactly two \textbf{S}-edges (i.e. edges labeled by a letter from \textbf{S}).\\
\end{itemize}

The figure below illustrates the definition of an $\textbf{S}$ band. In this figure, the edges $e,e_1,...,e_{n-1},f$ are \textbf{S}-edges, and the line $\ell(\pi_i,e_i)$ connects a fixed point in the interior of $\pi_i$ cells with a fixed point in the interior of $e_i$.\\

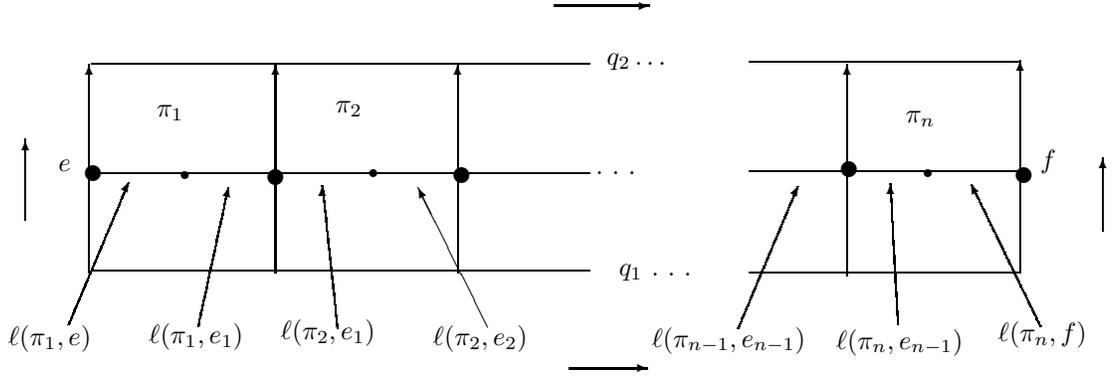
\begin{center}

\unitlength 1mm 
\linethickness{0.4pt}
\ifx\plotpoint\undefined\newsavebox{\plotpoint}\fi 
\begin{picture}(147.25,59)(0,0)
\put(12.75,23.75){\line(1,0){66.25}}
\put(12.75,51.25){\line(1,0){66.25}}
\put(12.75,36.75){\line(1,0){66.25}}
\put(12.5,23.5){\vector(0,1){27.75}}
\put(37.25,23.5){\vector(0,1){27.75}}
\put(61.5,23.5){\vector(0,1){27.75}}
\put(100.25,51.5){\line(1,0){36}}
\put(100.25,37){\line(1,0){36}}
\put(100.25,23.5){\line(1,0){36}}
\put(136.25,23.5){\vector(0,1){28.25}}
\put(113.25,23){\vector(0,1){28.25}}
\put(13,36.75){\circle*{2.062}}
\put(37.25,36.25){\circle*{2.062}}
\put(62,36.5){\circle*{2.062}}
\put(113.5,37.25){\circle*{2.062}}
\put(136.75,36.5){\circle*{2.062}}
\put(25.25,36.5){\circle*{1.118}}
\put(50.25,36.75){\circle*{1.118}}
\put(124,36.75){\circle*{1.118}}
\put(17.75,35.5){\vector(1,2){.07}}\multiput(9.75,16.5)(.033613445,.079831933){238}{\line(0,1){.079831933}}
\put(31,34.75){\vector(1,4){.07}}\multiput(27,16.25)(.033613445,.155462185){119}{\line(0,1){.155462185}}
\put(43.5,35){\vector(0,1){.07}}\multiput(45.5,16)(-.03333333,.31666667){60}{\line(0,1){.31666667}}
\put(106.75,35){\vector(1,2){.07}}\multiput(98.75,16)(.033613445,.079831933){238}{\line(0,1){.079831933}}
\put(119,35.25){\vector(0,1){.07}}\multiput(120,16.5)(-.0333333,.625){30}{\line(0,1){.625}}
\put(129.75,35.25){\vector(-1,2){.07}}\multiput(139,16.25)(-.0336363636,.0690909091){275}{\line(0,1){.0690909091}}
\put(65.5,16.25){\vector(-1,2){9.25}}
\put(85.25,51.5){\makebox(0,0)[cc]{$q_2 \,.\,.\,.$}}
\put(82.5,36.75){\makebox(0,0)[cc]{$.~~.~~.$}}
\put(87.75,23.5){\makebox(0,0)[cc]{$q_1~~.~~.~~.$}}
\put(7.25,14.75){\makebox(0,0)[cc]{$\ell(\pi_1,e)$}}
\put(26.75,15){\makebox(0,0)[cc]{$\ell(\pi_1,e_1)$}}
\put(44.25,15.5){\makebox(0,0)[cc]{$\ell(\pi_2,e_1)$}}
\put(64.5,14.75){\makebox(0,0)[cc]{$\ell(\pi_2,e_2)$}}
\put(97.5,14.5){\makebox(0,0)[cc]{$\ell(\pi_{n-1},e_{n-1})$}}
\put(120.25,14.25){\makebox(0,0)[cc]{$\ell(\pi_n,e_{n-1})$}}
\put(139.25,15.25){\makebox(0,0)[cc]{$\ell(\pi_n,f)$}}
\put(140,38.25){\makebox(0,0)[cc]{$f$}}
\put(9.25,37.75){\makebox(0,0)[cc]{$e$}}
\put(23.25,44.75){\makebox(0,0)[cc]{$\pi_1$}}
\put(47,45.25){\makebox(0,0)[cc]{$\pi_2$}}
\put(123,44){\makebox(0,0)[cc]{$\pi_n$}}
\put(74.25,59){\vector(1,0){12.75}}
\put(76.25,10.75){\vector(1,0){10.5}}
\put(147.25,29){\vector(0,1){9.75}}
\put(4,30.5){\vector(0,1){10.75}}
\end{picture}

\end{center}

The line formed by the segments $\ell(\pi_i,e_i),\ell(\pi_i,e_{i-1})$ connecting points inside neighboring cells is called the \textit{median} of the band $\mathcal{B}$. The \textbf{S}-edges $e$ and $f$ are called the \textit{start} and \textit{end} edges of the band.  If $\mathcal{B}=(\pi_1,...,\pi_n)$ is an \textbf{S}-band then $(\pi_n,\pi_{n-1},...,\pi_1)$ is also an \textbf{S}-band.  This band is called the {\em inverse} of $\mathcal{B}$ and is denoted $\mathcal{B}^{-1}$.  The start edge of $\mathcal{B}$ is the end edge of $\mathcal{B}^{-1}$, and the end edge of $\mathcal{B}$ is the start edge of $\mathcal{B}^{-1}$.\\

The clockwise boundary of the diagram formed by the cells $\pi_1,...,\pi_n$ of $\mathcal{B}$ has the form $e q_2 f^{-1} q_1^{-1}$. We call $q_1$ the \textit{bottom} of $\mathcal{B}$ and $q_2$ the \textit{top} of $\mathcal{B}$.  We denote these paths by \textbf{bot}$(\mathcal{B})$ and \textbf{top}$(\mathcal{B})$, respectively.\\

We say that two bands \textit{cross} if their medians cross. We say that a band is an \textit{annulus} if its median is a closed curve (i.e. if $e=f$).\\

We now define a type of surgery on diagrams called a {\em folding surgery}.  Suppose that $\Delta$ is a diagram and $p=e_1e_2^{-1}$ is a path in $\Delta$.  Suppose also that both $e_1$ and $e_2$ have label $x$, so the label of $p$ is $xx^{-1}$. Suppose that $v$ is the final vertex of $e_1$ and $e_2$ and that $e_1$ and $e_2$ do not share the same initial vertex. In order to perform a folding surgery on $\Delta$ at $p$, we first create a hole in $\Delta$ by cutting along $p$.  Creating this hole turns  $v$ into two new vertices, $v_1$ and $v_2$.  We then close the hole by folding edges together in such a way that $v_1$ and $v_2$ are not identified.\\

\section{The Group Presentation $P(M_{\infty})$}

The presentation $P(M_{\infty})$ from Theorem \ref{main} is constructed from the presentation $S(M'_{\infty})$ in exactly the same way that the presentation $P(M)$ mentioned in Theorem \ref{bigdeal} is constructed from $M$ in \cite{BRS}.  The generating set of $P(M_{\infty})$ includes all generators of $S(M'_{\infty})$ as well as the new generators $\{\kappa_i|i=1,...,2N\}$ for some sufficiently large $N$. We call these new generators $\kappa$-letters.  As in \cite{BRS} and \cite{BORS}, we use the notation $\bar{Y}=\hat{A}\cup \{\kappa_i|i=1,...,2N\}$\\

There are three types of relators of $P(M_{\infty})$: transition relators, auxiliary relators, and the hub relator.  The transition relators of $P(M_{\infty})$ are exactly the transition relators of $S(M'_{\infty})$.  The auxiliary relators of $S(M'_{\infty})$ consist of the auxiliary relators of $S(M'_{\infty})$ as well as the relators $\{\rho \kappa_i = \kappa_i \rho|i=1,...,2N \}$ for each command letter $\rho$ of $S(M'_{\infty})$.  If the boundary label of a 2-cell of a $P(M_{\infty})$ diagram is a transition relator, then we say that 2-cell is a {\em transition 2-cell}.  If the boundary label of a 2-cell of a $P(M_{\infty})$ diagram is an auxiliary relator, then we say that 2-cell is an {\em auxiliary 2-cell}.\\

For an admissible word $W$ of $S(M'_{\infty})$, let $K(W)$ denote the following word:

$$(W^{-1} \kappa_1 W \kappa_2 W^{-1} \kappa_3 W \kappa_4...W^{-1} \kappa_{2N-1} W \kappa_{2N})(\kappa_{2N} W^{-1} \kappa_{2N-1} W...\kappa_2 W^{-1} \kappa_1 W)^{-1}.$$\\

The hub relator is $K(W_0)=1$, where $W_0$ is the accept word of $S(M_{\infty})$.  These generators and relators constitute the presentation $P(M_{\infty})$.\\

For an admissible word $W$ of $S(M'_{\infty})$, the if the word $K(W)$ is trivial in $P(M'_{\infty})$ then $K(W)$ is called a {\em disc label} of $P(M'_{\infty})$. For every admissible word $W$ of $S(M'_{\infty})$, the word $K(W)$ is trivial in $P(M_{\infty})$ if and only if $W$ is accepted by $S(M'_{\infty})$. \\

We can now define the map $\mathcal{K}$, which has the same interpretation in both Theorem \ref{bigdeal} and Theorem \ref{smalldeal}.  If $\textbf{u}\in A^*$, and $\textbf{c}_{\textbf{u}}$ is the input configuration for $\textbf{u}$ in $M'_{\infty}$, then

\begin{equation}\label{kdef}
\mathcal{K}(\textbf{u}):=K(\sigma(\textbf{c}_{\textbf{u}})).
\end{equation}

\begin{lemma}\label{presce}

If $S(M'_{\infty})$ is c.e. then $P(M_{\infty})$ is c.e..

\end{lemma}

\begin{proof}

We construct $P(M_{\infty})$ from $S(M'_{\infty})$ by adding the $\kappa$ letters to the generating set of $S(M'_{\infty})$, and adding the hub relation and the relators $\{\rho \kappa_i = \kappa_i \rho|i=1,...,2N \}$ for each command letter $\rho$ of $S(M'_{\infty})$ to the relator set of $S(M'_{\infty})$.  Thus if $S(M'_{\infty})$ is c.e., so is $P(M'_{\infty})$.

\end{proof}

\section{The Word Problem For $P(M_{\infty})$}

Suppose that $M$ is a Turing machine that computes the word problem for a finitely generated group $G$.  In \cite{BORS}, the authors construct a finite presentation $H(M)$ of a group such that there is an embedding $G\to H(M)$.  The group presented by $H(M)$ is denoted in \cite{BORS} by $H_N(S)$ or often just $H$, while the group presented by $P(M)$ is denoted by $G_N(S)$.  The presentation $H(M)$ is obtained from $P(M)$ via a sequence of 3 HNN extensions, each of which adjoins finitely many stable letters to $P(M)$.  The set of generators of $H(M)$ contains the generators of $P(M)$ and the set of relators of $H(M)$ contains the set of relators of $P(M)$.  The structure of $H(M)$ diagrams is thoroughly analyzed in \cite{BORS}.  The proofs of the lemmas used in this analysis do not depend at all on the Turing machine $M$.  Instead, they rely on geometric properties of $H(M)$ that are invariant under the choice of $M$.\\

It is possible to construct $H(M_{\infty})$ from $P(M_{\infty})$ in in the exact same way that $H(M)$ is constructed from $P(M)$ in \cite{BORS}.  The results proven about $H(M)$ in \cite{BORS} hold for $H(M_{\infty})$ as well.  Furthermore, the proofs of these results for $H(M_{\infty})$ are identical to the corresponding proofs given in \cite{BORS}.  We will cite these results from \cite{BORS} as though they were statements about $H(M_{\infty})$ instead of $H(M)$.  \\

In \cite{BORS}, the authors analyze $H(M_{\infty})$ by constructing the {\em disc-based presentation for H(M)}.  This presentation is constructed by adding the disc labels of $P(M_{\infty})$ to the relator set of the presentation of $H(M_{\infty})$.  We denote the disc based presentation of $H(M_{\infty})$ by $H_D(M_{\infty})$.  We define the {\em disc-based presentation $P_D(M_{\infty})$} to be the presentation obtained by adding the disc labels of $P(M_{\infty})$ to the relator set of $P(M_{\infty})$.  If the boundary label of a 2-cell in a disc based presentation is a disc label, then we call that 2-cell a {\em disc}.\\

On page 486 of \cite{BORS}, the authors assign each diagram $\Delta$ over $H_D(M_{\infty})$ a 4-tuple of non-negative integers which they call the {\em type} of $\Delta$.  If $(n_1,n_2,n_3,n_4)$ is the type of $\Delta$, then each $n_i$ is the number of a certain kind of 2-cell in $\Delta$. For example, $n_1$ is the number of discs in $\Delta$.  They then order the types lexicographically.  A diagram $\Delta$ over $H_D(M_{\infty})$ is said to be {\em minimal} if the type of $\Delta$ is minimal among the types of all diagrams with the same boundary label as $\Delta$.  For our purposes, further details about the formal definitions in \cite{BORS} of the type of an $H_D(M_{\infty})$ diagram are unimportant.  We define a $P_D(M_{\infty})$ diagram $\Delta$ to be minimal if $\Delta$ is a minimal $H_D(M_{\infty})$ diagram.     \\

\begin{lemma}\label{useful}

Let $\Delta$ be a minimal diagram over $H_D(M_{\infty})$ such that $\partial\Delta$ is a word in the generators of $P(M_{\infty})$; then $\Delta$ is a diagram over $P_D(M_{\infty})$.

\end{lemma}

\begin{proof}

This is \cite[Lemma 4.1]{BORS}.

\end{proof}

\begin{corollary}\label{usefulcor}

For every trivial word $w$ in the generators of $P(M_{\infty})$ there is a minimal $P_D(M_{\infty})$ diagram with boundary label $w$.

\end{corollary}

\begin{lemma}\label{noThetaAn}

Minimal diagrams over $P_D(M_{\infty})$ contain no $\hat{\Theta}$ annuli.

\end{lemma}

\begin{proof}

This is \cite[Lemma 4.25]{BORS}.

\end{proof}

\begin{lemma}\label{kappahub}

Suppose $\Delta$ is a minimal diagram over the disc based presentation $P_D(M_{\infty})$.  If $\Delta$ contains at least one disc then there exists a disc $\Pi$ in $\Delta$ with $4N-6$ consecutive $\kappa$-bands $\mathcal{B}_1,\dots, \mathcal{B}_{4N-6}$ starting on $\partial\Pi$ and ending on $\partial\Delta$. For every such disc, let $\Phi_{\Delta}(\Pi)$ be the subdiagram of $\Delta$ bounded by $\textbf{top}(\mathcal{B}_1), \textbf{bot}(\mathcal{B}_{4N-6})$, $\partial\Delta$ and $\partial\Pi$, which contains $\mathcal{B}_1, \mathcal{B}_{N-6}$ and does not contain $\Pi$ (there is only one subdiagram in $\Delta$ satisfying these conditions). Then there exists a disc $\Pi$ such that $\Phi_{\Delta}(\Pi)$ does not contain discs.

\end{lemma}

\begin{proof}

This lemma follows from \cite[Lemma 4.21]{BORS} and \cite[Lemma 4.24]{BORS}.

\end{proof}

\begin{lemma}\label{barY}

In a minimal diagram $\Delta$ over $P_D(M_{\infty})$, a $\bar{Y}$ band can not begin and end on the same disc.  If $\Delta$ contains no discs, then $\Delta$ contains no $\bar{Y}$ annuli.

\end{lemma}

\begin{proof}

This follows from \cite[Lemma 4.30]{BORS} and \cite[Lemma 4.32]{BORS}.

\end{proof}

\begin{lemma}\label{mindiagbounds}

If $\Delta$ is a minimal diagram over $P_D(M_{\infty})$, then the number of $\hat{\Theta}$ 2-cells in $\Delta$ is $O(|\partial\Delta|^3)$.  Also, the sum of the boundary lengths of the discs in $\Delta$ is $O(|\partial\Delta|^2)$

\end{lemma}

\begin{proof}

This follows from \cite[Lemma 5.10]{BORS} and \cite[Lemma 5.14]{BORS}.\\

\end{proof}

We now prove that if $M_{\infty}$ is a c.e. union machine that accepts a decidable language, then the word problem for $P(M_{\infty})$ is solvable.

\begin{lemma}\label{mathcalt}

For a word $w$ in $P(M_{\infty})$, we can effectively construct a finite set $\mathcal{T}(w)$ of transition relators of $P(M_{\infty})$ such that if $w$ is trivial in $P(M_{\infty})$ with minimal $P_D(M_{\infty})$ diagram $\Delta$, then the label of every transition 2-cell in $\Delta$ is in $\mathcal{T}(w)$.

\end{lemma}

\begin{proof}

By Lemma \ref{noThetaAn} there are no $\hat{\Theta}$ annuli in $\Delta$.  Thus every $\hat{\Theta}$ band in $\Delta$ must both start and end on the boundary of $\Delta$.  The lemma now follows from Lemma \ref{comenu}.

\end{proof}

For a word $w$ in $P(M_{\infty})$, let $\mathcal{A}(w)$ be the set of $\hat{A}$ letters that appear in $w$ or in an element of $\mathcal{T}(w)$.  Let $\mathcal{Q}(w)$ be the set of $\hat{Q}$ letters that appear in $w$ or in an element of $\mathcal{T}(w)$.\\

\begin{lemma}\label{mathcala}

If $w$ is trivial in $P(M_{\infty})$ with minimal $P_D(M_{\infty})$ diagram $\Delta$ then the following statements hold:

\begin{enumerate}
\item
Every $\hat{A}$ letter that labels an edge of $\Delta$ is contained in $\mathcal{A}(w)$.

\item
Every $\hat{Q}$ letter that labels an edge of $\Delta$ is contained in $\mathcal{Q}(w)$.

\end{enumerate}

\end{lemma}

\begin{proof}

We prove this Lemma by induction on the number of discs in $\Delta$.  If $\Delta$ contains zero discs then part 1 follows from Lemma \ref{mathcalt} and the fact that every $\hat{Y}$-band in $\Delta$ begins (ends) either on $\partial\Delta$ or on the boundary of a transition cell of $\Delta$. Part 2 follows from the fact that every 2-cell in $\Delta$ whose boundary contains a $\hat{Q}$ edge is a transition 2-cell of $\Delta$.\\

If $\Delta$ contains $n>0$ discs, then by Lemma \ref{kappahub} there exists a disc $\Pi$ in $\Delta$ such that $\Phi_N(\Pi)$ contains no discs.  Since $\Phi_N(\Pi)$ is bounded by $\textbf{top}(\mathcal{B}_1), \textbf{bot}(\mathcal{B}_{4N-6})$, $\partial\Delta$ and $\partial\Pi$, the boundary of $\Phi_N(\Pi)$ is $p_1q_1s_1^{-1}q_2$ where $p_1$ is a subpath of the boundary of $\Delta$, $q_1,q_2$ are $\hat{\Theta}$ paths, and $s_1$ is a subpath of the boundary of $\Pi$.  The boundary of $\Pi$ is $s_1s_2$, and the boundary of $\Delta$ is $p_1p_2$\\

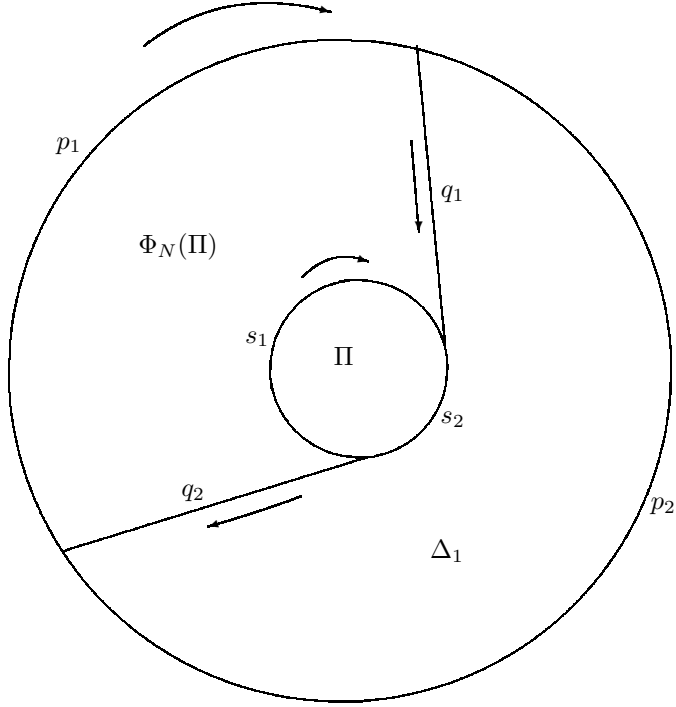
\begin{center}

\unitlength 1mm 
\linethickness{0.4pt}
\ifx\plotpoint\undefined\newsavebox{\plotpoint}\fi 
\begin{picture}(92.03,98.5)(0,0)
\put(92.03,47){\line(0,1){1.5027}}
\put(92.004,48.503){\line(0,1){1.501}}
\multiput(91.927,50.004)(-.03203,.37437){4}{\line(0,1){.37437}}
\multiput(91.799,51.501)(-.029857,.248707){6}{\line(0,1){.248707}}
\multiput(91.62,52.993)(-.032853,.21218){7}{\line(0,1){.21218}}
\multiput(91.39,54.479)(-.03117,.164061){9}{\line(0,1){.164061}}
\multiput(91.109,55.955)(-.033076,.146611){10}{\line(0,1){.146611}}
\multiput(90.779,57.421)(-.031717,.121164){12}{\line(0,1){.121164}}
\multiput(90.398,58.875)(-.0330775,.1107794){13}{\line(0,1){.1107794}}
\multiput(89.968,60.315)(-.0319272,.0949744){15}{\line(0,1){.0949744}}
\multiput(89.489,61.74)(-.0329532,.0879651){16}{\line(0,1){.0879651}}
\multiput(88.962,63.148)(-.0319434,.0771459){18}{\line(0,1){.0771459}}
\multiput(88.387,64.536)(-.0327389,.0720102){19}{\line(0,1){.0720102}}
\multiput(87.765,65.904)(-.0334187,.0673083){20}{\line(0,1){.0673083}}
\multiput(87.097,67.251)(-.0324513,.0601168){22}{\line(0,1){.0601168}}
\multiput(86.383,68.573)(-.0329849,.0564101){23}{\line(0,1){.0564101}}
\multiput(85.624,69.871)(-.0334372,.0529493){24}{\line(0,1){.0529493}}
\multiput(84.821,71.141)(-.0325152,.0477944){26}{\line(0,1){.0477944}}
\multiput(83.976,72.384)(-.0328635,.0449288){27}{\line(0,1){.0449288}}
\multiput(83.089,73.597)(-.03315,.0422174){28}{\line(0,1){.0422174}}
\multiput(82.161,74.779)(-.0333795,.0396455){29}{\line(0,1){.0396455}}
\multiput(81.193,75.929)(-.033556,.0372003){30}{\line(0,1){.0372003}}
\multiput(80.186,77.045)(-.0336834,.034871){31}{\line(0,1){.034871}}
\multiput(79.142,78.126)(-.0348539,.0337011){31}{\line(-1,0){.0348539}}
\multiput(78.061,79.171)(-.0371833,.033575){30}{\line(-1,0){.0371833}}
\multiput(76.946,80.178)(-.0396285,.0333997){29}{\line(-1,0){.0396285}}
\multiput(75.797,81.146)(-.0422005,.0331715){28}{\line(-1,0){.0422005}}
\multiput(74.615,82.075)(-.0449121,.0328864){27}{\line(-1,0){.0449121}}
\multiput(73.402,82.963)(-.0477779,.0325396){26}{\line(-1,0){.0477779}}
\multiput(72.16,83.809)(-.0529323,.0334641){24}{\line(-1,0){.0529323}}
\multiput(70.89,84.612)(-.0563933,.0330136){23}{\line(-1,0){.0563933}}
\multiput(69.593,85.372)(-.0601003,.0324819){22}{\line(-1,0){.0601003}}
\multiput(68.27,86.086)(-.0672913,.0334529){20}{\line(-1,0){.0672913}}
\multiput(66.925,86.755)(-.0719935,.0327756){19}{\line(-1,0){.0719935}}
\multiput(65.557,87.378)(-.0771296,.0319826){18}{\line(-1,0){.0771296}}
\multiput(64.168,87.954)(-.0879483,.032998){16}{\line(-1,0){.0879483}}
\multiput(62.761,88.482)(-.0949582,.0319756){15}{\line(-1,0){.0949582}}
\multiput(61.337,88.961)(-.1107625,.0331339){13}{\line(-1,0){.1107625}}
\multiput(59.897,89.392)(-.121148,.031779){12}{\line(-1,0){.121148}}
\multiput(58.443,89.773)(-.146594,.033151){10}{\line(-1,0){.146594}}
\multiput(56.977,90.105)(-.164045,.031253){9}{\line(-1,0){.164045}}
\multiput(55.501,90.386)(-.212163,.032961){7}{\line(-1,0){.212163}}
\multiput(54.016,90.617)(-.248692,.029984){6}{\line(-1,0){.248692}}
\multiput(52.524,90.797)(-.37436,.03222){4}{\line(-1,0){.37436}}
\put(51.026,90.926){\line(-1,0){1.5009}}
\put(49.525,91.003){\line(-1,0){3.0055}}
\put(46.52,91.005){\line(-1,0){1.501}}
\multiput(45.019,90.929)(-.37439,-.03184){4}{\line(-1,0){.37439}}
\multiput(43.521,90.801)(-.248723,-.029731){6}{\line(-1,0){.248723}}
\multiput(42.029,90.623)(-.212197,-.032745){7}{\line(-1,0){.212197}}
\multiput(40.543,90.394)(-.164077,-.031086){9}{\line(-1,0){.164077}}
\multiput(39.067,90.114)(-.146628,-.033001){10}{\line(-1,0){.146628}}
\multiput(37.6,89.784)(-.12118,-.031655){12}{\line(-1,0){.12118}}
\multiput(36.146,89.404)(-.1107962,-.0330211){13}{\line(-1,0){.1107962}}
\multiput(34.706,88.975)(-.0949907,-.0318789){15}{\line(-1,0){.0949907}}
\multiput(33.281,88.497)(-.0879818,-.0329085){16}{\line(-1,0){.0879818}}
\multiput(31.873,87.97)(-.0771621,-.0319041){18}{\line(-1,0){.0771621}}
\multiput(30.484,87.396)(-.0720268,-.0327023){19}{\line(-1,0){.0720268}}
\multiput(29.116,86.775)(-.0673253,-.0333844){20}{\line(-1,0){.0673253}}
\multiput(27.769,86.107)(-.0601333,-.0324207){22}{\line(-1,0){.0601333}}
\multiput(26.446,85.394)(-.0564269,-.0329562){23}{\line(-1,0){.0564269}}
\multiput(25.149,84.636)(-.0529663,-.0334102){24}{\line(-1,0){.0529663}}
\multiput(23.877,83.834)(-.047811,-.0324909){26}{\line(-1,0){.047811}}
\multiput(22.634,82.989)(-.0449455,-.0328407){27}{\line(-1,0){.0449455}}
\multiput(21.421,82.102)(-.0422342,-.0331286){28}{\line(-1,0){.0422342}}
\multiput(20.238,81.175)(-.0396624,-.0333593){29}{\line(-1,0){.0396624}}
\multiput(19.088,80.207)(-.0372174,-.0335371){30}{\line(-1,0){.0372174}}
\multiput(17.972,79.201)(-.0348882,-.0336656){31}{\line(-1,0){.0348882}}
\multiput(16.89,78.158)(-.0337188,-.0348367){31}{\line(0,-1){.0348367}}
\multiput(15.845,77.078)(-.0335939,-.0371662){30}{\line(0,-1){.0371662}}
\multiput(14.837,75.963)(-.0334198,-.0396115){29}{\line(0,-1){.0396115}}
\multiput(13.868,74.814)(-.033193,-.0421836){28}{\line(0,-1){.0421836}}
\multiput(12.938,73.633)(-.0329093,-.0448953){27}{\line(0,-1){.0448953}}
\multiput(12.05,72.421)(-.0325639,-.0477613){26}{\line(0,-1){.0477613}}
\multiput(11.203,71.179)(-.033491,-.0529153){24}{\line(0,-1){.0529153}}
\multiput(10.399,69.909)(-.0330423,-.0563765){23}{\line(0,-1){.0563765}}
\multiput(9.639,68.612)(-.0325125,-.0600837){22}{\line(0,-1){.0600837}}
\multiput(8.924,67.29)(-.0334872,-.0672742){20}{\line(0,-1){.0672742}}
\multiput(8.254,65.945)(-.0328122,-.0719768){19}{\line(0,-1){.0719768}}
\multiput(7.631,64.577)(-.0320219,-.0771133){18}{\line(0,-1){.0771133}}
\multiput(7.055,63.189)(-.0330428,-.0879315){16}{\line(0,-1){.0879315}}
\multiput(6.526,61.782)(-.0320239,-.0949419){15}{\line(0,-1){.0949419}}
\multiput(6.045,60.358)(-.0331902,-.1107456){13}{\line(0,-1){.1107456}}
\multiput(5.614,58.919)(-.03184,-.121132){12}{\line(0,-1){.121132}}
\multiput(5.232,57.465)(-.033225,-.146577){10}{\line(0,-1){.146577}}
\multiput(4.9,55.999)(-.031337,-.164029){9}{\line(0,-1){.164029}}
\multiput(4.618,54.523)(-.033069,-.212147){7}{\line(0,-1){.212147}}
\multiput(4.386,53.038)(-.030111,-.248677){6}{\line(0,-1){.248677}}
\multiput(4.205,51.546)(-.03241,-.37434){4}{\line(0,-1){.37434}}
\put(4.076,50.048){\line(0,-1){1.5009}}
\put(3.997,48.548){\line(0,-1){4.5065}}
\multiput(4.07,44.041)(.03165,-.3744){4}{\line(0,-1){.3744}}
\multiput(4.196,42.543)(.029604,-.248738){6}{\line(0,-1){.248738}}
\multiput(4.374,41.051)(.032637,-.212213){7}{\line(0,-1){.212213}}
\multiput(4.602,39.565)(.031003,-.164092){9}{\line(0,-1){.164092}}
\multiput(4.881,38.089)(.032927,-.146645){10}{\line(0,-1){.146645}}
\multiput(5.211,36.622)(.031594,-.121196){12}{\line(0,-1){.121196}}
\multiput(5.59,35.168)(.0329647,-.110813){13}{\line(0,-1){.110813}}
\multiput(6.018,33.727)(.0318306,-.0950069){15}{\line(0,-1){.0950069}}
\multiput(6.496,32.302)(.0328637,-.0879986){16}{\line(0,-1){.0879986}}
\multiput(7.022,30.894)(.0337392,-.0817183){17}{\line(0,-1){.0817183}}
\multiput(7.595,29.505)(.0326656,-.0720434){19}{\line(0,-1){.0720434}}
\multiput(8.216,28.136)(.0333502,-.0673423){20}{\line(0,-1){.0673423}}
\multiput(8.883,26.789)(.0323901,-.0601498){22}{\line(0,-1){.0601498}}
\multiput(9.595,25.466)(.0329275,-.0564437){23}{\line(0,-1){.0564437}}
\multiput(10.353,24.168)(.0333833,-.0529833){24}{\line(0,-1){.0529833}}
\multiput(11.154,22.896)(.0324666,-.0478275){26}{\line(0,-1){.0478275}}
\multiput(11.998,21.653)(.0328178,-.0449622){27}{\line(0,-1){.0449622}}
\multiput(12.884,20.439)(.0331071,-.0422511){28}{\line(0,-1){.0422511}}
\multiput(13.811,19.256)(.0333391,-.0396794){29}{\line(0,-1){.0396794}}
\multiput(14.778,18.105)(.0335182,-.0372345){30}{\line(0,-1){.0372345}}
\multiput(15.784,16.988)(.0336478,-.0349053){31}{\line(0,-1){.0349053}}
\multiput(16.827,15.906)(.0348196,-.0337366){31}{\line(1,0){.0348196}}
\multiput(17.906,14.86)(.0371491,-.0336128){30}{\line(1,0){.0371491}}
\multiput(19.021,13.852)(.0395944,-.03344){29}{\line(1,0){.0395944}}
\multiput(20.169,12.882)(.0421667,-.0332145){28}{\line(1,0){.0421667}}
\multiput(21.349,11.952)(.0448786,-.0329321){27}{\line(1,0){.0448786}}
\multiput(22.561,11.063)(.0477447,-.0325882){26}{\line(1,0){.0477447}}
\multiput(23.802,10.215)(.0528982,-.033518){24}{\line(1,0){.0528982}}
\multiput(25.072,9.411)(.0563597,-.033071){23}{\line(1,0){.0563597}}
\multiput(26.368,8.65)(.0600672,-.0325431){22}{\line(1,0){.0600672}}
\multiput(27.69,7.934)(.0672572,-.0335214){20}{\line(1,0){.0672572}}
\multiput(29.035,7.264)(.0719601,-.0328488){19}{\line(1,0){.0719601}}
\multiput(30.402,6.64)(.077097,-.0320611){18}{\line(1,0){.077097}}
\multiput(31.79,6.063)(.0879147,-.0330875){16}{\line(1,0){.0879147}}
\multiput(33.197,5.533)(.0949256,-.0320722){15}{\line(1,0){.0949256}}
\multiput(34.62,5.052)(.1107287,-.0332466){13}{\line(1,0){.1107287}}
\multiput(36.06,4.62)(.121115,-.031902){12}{\line(1,0){.121115}}
\multiput(37.513,4.237)(.146561,-.0333){10}{\line(1,0){.146561}}
\multiput(38.979,3.904)(.164013,-.03142){9}{\line(1,0){.164013}}
\multiput(40.455,3.621)(.21213,-.033177){7}{\line(1,0){.21213}}
\multiput(41.94,3.389)(.248661,-.030237){6}{\line(1,0){.248661}}
\multiput(43.432,3.208)(.37432,-.0326){4}{\line(1,0){.37432}}
\put(44.929,3.077){\line(1,0){1.5009}}
\put(46.43,2.998){\line(1,0){3.0055}}
\put(49.436,2.994){\line(1,0){1.5011}}
\multiput(50.937,3.068)(.37442,.03146){4}{\line(1,0){.37442}}
\multiput(52.434,3.194)(.248753,.029478){6}{\line(1,0){.248753}}
\multiput(53.927,3.371)(.21223,.032529){7}{\line(1,0){.21223}}
\multiput(55.412,3.599)(.164108,.030919){9}{\line(1,0){.164108}}
\multiput(56.889,3.877)(.146662,.032852){10}{\line(1,0){.146662}}
\multiput(58.356,4.205)(.121212,.031532){12}{\line(1,0){.121212}}
\multiput(59.811,4.584)(.1108297,.0329083){13}{\line(1,0){.1108297}}
\multiput(61.251,5.012)(.0950231,.0317822){15}{\line(1,0){.0950231}}
\multiput(62.677,5.488)(.0880153,.0328189){16}{\line(1,0){.0880153}}
\multiput(64.085,6.013)(.0817354,.0336976){17}{\line(1,0){.0817354}}
\multiput(65.474,6.586)(.0720601,.0326289){19}{\line(1,0){.0720601}}
\multiput(66.844,7.206)(.0673592,.0333159){20}{\line(1,0){.0673592}}
\multiput(68.191,7.873)(.0601663,.0323595){22}{\line(1,0){.0601663}}
\multiput(69.514,8.584)(.0564604,.0328987){23}{\line(1,0){.0564604}}
\multiput(70.813,9.341)(.0530003,.0333563){24}{\line(1,0){.0530003}}
\multiput(72.085,10.142)(.0497578,.0337399){25}{\line(1,0){.0497578}}
\multiput(73.329,10.985)(.0449789,.0327949){27}{\line(1,0){.0449789}}
\multiput(74.543,11.871)(.0422679,.0330855){28}{\line(1,0){.0422679}}
\multiput(75.727,12.797)(.0396964,.0333189){29}{\line(1,0){.0396964}}
\multiput(76.878,13.763)(.0372515,.0334992){30}{\line(1,0){.0372515}}
\multiput(77.996,14.768)(.0349224,.0336301){31}{\line(1,0){.0349224}}
\multiput(79.078,15.811)(.0326995,.0337148){32}{\line(0,1){.0337148}}
\multiput(80.125,16.89)(.0336317,.0371319){30}{\line(0,1){.0371319}}
\multiput(81.134,18.004)(.0334601,.0395774){29}{\line(0,1){.0395774}}
\multiput(82.104,19.151)(.0332359,.0421498){28}{\line(0,1){.0421498}}
\multiput(83.035,20.332)(.0329549,.0448618){27}{\line(0,1){.0448618}}
\multiput(83.924,21.543)(.0326125,.0477281){26}{\line(0,1){.0477281}}
\multiput(84.772,22.784)(.0335449,.0528812){24}{\line(0,1){.0528812}}
\multiput(85.577,24.053)(.0330997,.0563429){23}{\line(0,1){.0563429}}
\multiput(86.339,25.349)(.0325736,.0600506){22}{\line(0,1){.0600506}}
\multiput(87.055,26.67)(.0335557,.0672401){20}{\line(0,1){.0672401}}
\multiput(87.726,28.015)(.0328855,.0719434){19}{\line(0,1){.0719434}}
\multiput(88.351,29.382)(.0321004,.0770807){18}{\line(0,1){.0770807}}
\multiput(88.929,30.769)(.0331323,.0878978){16}{\line(0,1){.0878978}}
\multiput(89.459,32.175)(.0321205,.0949092){15}{\line(0,1){.0949092}}
\multiput(89.941,33.599)(.033303,.1107118){13}{\line(0,1){.1107118}}
\multiput(90.374,35.038)(.031964,.121099){12}{\line(0,1){.121099}}
\multiput(90.757,36.492)(.033374,.146544){10}{\line(0,1){.146544}}
\multiput(91.091,37.957)(.031504,.163997){9}{\line(0,1){.163997}}
\multiput(91.375,39.433)(.033285,.212113){7}{\line(0,1){.212113}}
\multiput(91.608,40.918)(.030364,.248646){6}{\line(0,1){.248646}}
\multiput(91.79,42.41)(.03279,.37431){4}{\line(0,1){.37431}}
\put(91.921,43.907){\line(0,1){1.5008}}
\put(92.001,45.408){\line(0,1){1.5923}}
\put(62.292,47.25){\line(0,1){.6193}}
\put(62.276,47.869){\line(0,1){.6176}}
\put(62.227,48.487){\line(0,1){.6142}}
\put(62.146,49.101){\line(0,1){.6091}}
\multiput(62.033,49.71)(-.029024,.120458){5}{\line(0,1){.120458}}
\multiput(61.888,50.312)(-.029425,.098973){6}{\line(0,1){.098973}}
\multiput(61.711,50.906)(-.029642,.083392){7}{\line(0,1){.083392}}
\multiput(61.504,51.49)(-.029733,.071505){8}{\line(0,1){.071505}}
\multiput(61.266,52.062)(-.033447,.069845){8}{\line(0,1){.069845}}
\multiput(60.998,52.621)(-.03295,.060437){9}{\line(0,1){.060437}}
\multiput(60.702,53.165)(-.032471,.052761){10}{\line(0,1){.052761}}
\multiput(60.377,53.692)(-.031997,.046348){11}{\line(0,1){.046348}}
\multiput(60.025,54.202)(-.031522,.040887){12}{\line(0,1){.040887}}
\multiput(59.647,54.693)(-.033625,.039175){12}{\line(0,1){.039175}}
\multiput(59.243,55.163)(-.032895,.0344817){13}{\line(0,1){.0344817}}
\multiput(58.816,55.611)(-.0346605,.0327065){13}{\line(-1,0){.0346605}}
\multiput(58.365,56.036)(-.039358,.033411){12}{\line(-1,0){.039358}}
\multiput(57.893,56.437)(-.041058,.031298){12}{\line(-1,0){.041058}}
\multiput(57.4,56.813)(-.046522,.031744){11}{\line(-1,0){.046522}}
\multiput(56.888,57.162)(-.052938,.032183){10}{\line(-1,0){.052938}}
\multiput(56.359,57.484)(-.060616,.03262){9}{\line(-1,0){.060616}}
\multiput(55.814,57.778)(-.070026,.033066){8}{\line(-1,0){.070026}}
\multiput(55.253,58.042)(-.081904,.033534){7}{\line(-1,0){.081904}}
\multiput(54.68,58.277)(-.083552,.029187){7}{\line(-1,0){.083552}}
\multiput(54.095,58.481)(-.099132,.028885){6}{\line(-1,0){.099132}}
\multiput(53.5,58.654)(-.120614,.028366){5}{\line(-1,0){.120614}}
\put(52.897,58.796){\line(-1,0){.6097}}
\put(52.288,58.906){\line(-1,0){.6146}}
\put(51.673,58.984){\line(-1,0){.6179}}
\put(51.055,59.029){\line(-1,0){1.2386}}
\put(49.816,59.023){\line(-1,0){.6173}}
\put(49.199,58.971){\line(-1,0){.6137}}
\put(48.585,58.886){\line(-1,0){.6085}}
\multiput(47.977,58.769)(-.120298,-.02968){5}{\line(-1,0){.120298}}
\multiput(47.375,58.621)(-.098811,-.029964){6}{\line(-1,0){.098811}}
\multiput(46.783,58.441)(-.083229,-.030096){7}{\line(-1,0){.083229}}
\multiput(46.2,58.231)(-.071342,-.030122){8}{\line(-1,0){.071342}}
\multiput(45.629,57.99)(-.061922,-.030069){9}{\line(-1,0){.061922}}
\multiput(45.072,57.719)(-.060257,-.033279){9}{\line(-1,0){.060257}}
\multiput(44.53,57.419)(-.052583,-.032758){10}{\line(-1,0){.052583}}
\multiput(44.004,57.092)(-.046173,-.032249){11}{\line(-1,0){.046173}}
\multiput(43.496,56.737)(-.040715,-.031744){12}{\line(-1,0){.040715}}
\multiput(43.007,56.356)(-.0359919,-.0312354){13}{\line(-1,0){.0359919}}
\multiput(42.539,55.95)(-.0343018,-.0330825){13}{\line(-1,0){.0343018}}
\multiput(42.094,55.52)(-.0325171,-.0348383){13}{\line(0,-1){.0348383}}
\multiput(41.671,55.067)(-.033196,-.039539){12}{\line(0,-1){.039539}}
\multiput(41.272,54.593)(-.031074,-.041228){12}{\line(0,-1){.041228}}
\multiput(40.9,54.098)(-.03149,-.046694){11}{\line(0,-1){.046694}}
\multiput(40.553,53.584)(-.031894,-.053112){10}{\line(0,-1){.053112}}
\multiput(40.234,53.053)(-.032289,-.060793){9}{\line(0,-1){.060793}}
\multiput(39.944,52.506)(-.032683,-.070206){8}{\line(0,-1){.070206}}
\multiput(39.682,51.944)(-.033087,-.082086){7}{\line(0,-1){.082086}}
\multiput(39.451,51.37)(-.033519,-.097662){6}{\line(0,-1){.097662}}
\multiput(39.249,50.784)(-.028344,-.099288){6}{\line(0,-1){.099288}}
\multiput(39.079,50.188)(-.027708,-.120767){5}{\line(0,-1){.120767}}
\put(38.941,49.584){\line(0,-1){.6103}}
\put(38.834,48.974){\line(0,-1){.615}}
\put(38.76,48.359){\line(0,-1){1.2375}}
\put(38.708,47.121){\line(0,-1){.6191}}
\put(38.731,46.502){\line(0,-1){.617}}
\put(38.787,45.885){\line(0,-1){.6133}}
\put(38.875,45.272){\line(0,-1){.6078}}
\multiput(38.995,44.664)(.030335,-.120134){5}{\line(0,-1){.120134}}
\multiput(39.146,44.064)(.030502,-.098646){6}{\line(0,-1){.098646}}
\multiput(39.329,43.472)(.030549,-.083064){7}{\line(0,-1){.083064}}
\multiput(39.543,42.89)(.030511,-.071177){8}{\line(0,-1){.071177}}
\multiput(39.787,42.321)(.030406,-.061757){9}{\line(0,-1){.061757}}
\multiput(40.061,41.765)(.033607,-.060075){9}{\line(0,-1){.060075}}
\multiput(40.363,41.224)(.033044,-.052404){10}{\line(0,-1){.052404}}
\multiput(40.694,40.7)(.032501,-.045997){11}{\line(0,-1){.045997}}
\multiput(41.051,40.194)(.031965,-.040541){12}{\line(0,-1){.040541}}
\multiput(41.435,39.708)(.0314312,-.0358211){13}{\line(0,-1){.0358211}}
\multiput(41.843,39.242)(.033269,-.034121){13}{\line(0,-1){.034121}}
\multiput(42.276,38.799)(.0350151,-.0323267){13}{\line(1,0){.0350151}}
\multiput(42.731,38.378)(.03972,-.03298){12}{\line(1,0){.03972}}
\multiput(43.208,37.983)(.04516,-.033653){11}{\line(1,0){.04516}}
\multiput(43.704,37.612)(.046865,-.031235){11}{\line(1,0){.046865}}
\multiput(44.22,37.269)(.053285,-.031604){10}{\line(1,0){.053285}}
\multiput(44.753,36.953)(.060968,-.031957){9}{\line(1,0){.060968}}
\multiput(45.302,36.665)(.070383,-.0323){8}{\line(1,0){.070383}}
\multiput(45.865,36.407)(.082265,-.032639){7}{\line(1,0){.082265}}
\multiput(46.44,36.178)(.097843,-.032986){6}{\line(1,0){.097843}}
\multiput(47.028,35.98)(.119329,-.033362){5}{\line(1,0){.119329}}
\multiput(47.624,35.814)(.120916,-.02705){5}{\line(1,0){.120916}}
\put(48.229,35.678){\line(1,0){.6108}}
\put(48.84,35.575){\line(1,0){.6154}}
\put(49.455,35.504){\line(1,0){.6183}}
\put(50.073,35.465){\line(1,0){.6195}}
\put(50.693,35.459){\line(1,0){.619}}
\put(51.312,35.486){\line(1,0){.6167}}
\put(51.929,35.544){\line(1,0){.6128}}
\multiput(52.541,35.636)(.15179,.03081){4}{\line(1,0){.15179}}
\multiput(53.148,35.759)(.119967,.03099){5}{\line(1,0){.119967}}
\multiput(53.748,35.914)(.098478,.031039){6}{\line(1,0){.098478}}
\multiput(54.339,36.1)(.082896,.031001){7}{\line(1,0){.082896}}
\multiput(54.919,36.317)(.071009,.030898){8}{\line(1,0){.071009}}
\multiput(55.488,36.564)(.06159,.030742){9}{\line(1,0){.06159}}
\multiput(56.042,36.841)(.053901,.030541){10}{\line(1,0){.053901}}
\multiput(56.581,37.146)(.052223,.033329){10}{\line(1,0){.052223}}
\multiput(57.103,37.48)(.045819,.032751){11}{\line(1,0){.045819}}
\multiput(57.607,37.84)(.040366,.032186){12}{\line(1,0){.040366}}
\multiput(58.091,38.226)(.0356492,.031626){13}{\line(1,0){.0356492}}
\multiput(58.555,38.637)(.0339391,.0334545){13}{\line(1,0){.0339391}}
\multiput(58.996,39.072)(.0321353,.0351908){13}{\line(0,1){.0351908}}
\multiput(59.414,39.53)(.032763,.039899){12}{\line(0,1){.039899}}
\multiput(59.807,40.008)(.033406,.045343){11}{\line(0,1){.045343}}
\multiput(60.175,40.507)(.030979,.047035){11}{\line(0,1){.047035}}
\multiput(60.515,41.025)(.031313,.053457){10}{\line(0,1){.053457}}
\multiput(60.828,41.559)(.031624,.061142){9}{\line(0,1){.061142}}
\multiput(61.113,42.109)(.031916,.070558){8}{\line(0,1){.070558}}
\multiput(61.368,42.674)(.03219,.082442){7}{\line(0,1){.082442}}
\multiput(61.594,43.251)(.032452,.098022){6}{\line(0,1){.098022}}
\multiput(61.788,43.839)(.032711,.119509){5}{\line(0,1){.119509}}
\multiput(61.952,44.437)(.03299,.15133){4}{\line(0,1){.15133}}
\put(62.084,45.042){\line(0,1){.6114}}
\put(62.184,45.653){\line(0,1){.6158}}
\put(62.252,46.269){\line(0,1){.9809}}
\multiput(58.25,90.25)(.033482143,-.357142857){112}{\line(0,-1){.357142857}}
\multiput(51.75,35.5)(-.1091644205,-.0336927224){371}{\line(-1,0){.1091644205}}
\put(48.5,49){\makebox(0,0)[cc]{$\Pi$}}
\put(26.5,63.5){\makebox(0,0)[cc]{$\Phi_N(\Pi)$}}
\put(12,77){\makebox(0,0)[cc]{$p_1$}}
\put(91,29){\makebox(0,0)[cc]{$p_2$}}
\put(37,51.25){\makebox(0,0)[cc]{$s_1$}}
\put(63,40.5){\makebox(0,0)[cc]{$s_2$}}
\put(28.5,30.75){\makebox(0,0)[cc]{$q_2$}}
\put(63,70.5){\makebox(0,0)[cc]{$q_1$}}
\put(46.75,94.75){\vector(4,-1){.07}}\qbezier(22,90.25)(32.625,98.5)(46.75,94.75)
\put(58.5,65.5){\vector(0,-1){.07}}\qbezier(57.5,77.5)(58,71)(58.5,65.5)
\put(30.5,26.25){\vector(-3,-1){.07}}\qbezier(42.75,30.25)(37.375,28.25)(30.5,26.25)
\put(51.75,61.5){\vector(3,-1){.07}}\qbezier(43,59.5)(47.125,63.5)(51.75,61.5)
\put(62.25,23){\makebox(0,0)[cc]{$\Delta_1$}}
\end{picture}

\end{center}

We first prove part 1.  By Lemma \ref{barY} and the fact that $q_1,q_2$ are $\hat{\Theta}$ bands, every $\bar{Y}$ band in $\Phi_N(\Pi)$ that starts on $s_1$ must end either on $p_1$ or on a transition 2-cell of $\Phi_N(\Pi)$.  Therefore every $\hat{A}$ letter that labels an edge of $\Phi_N(\Pi)$ either appears in $\text{Lab}(p_1)$ (a subword of $w$) or appears in an element of $\mathcal{T}(w)$.\\

By Lemma \ref{kappahub}, $s_1$ contains more than $\frac{1}{2}$ of the length of $\partial\Pi$.  Thus, by the definition of a disc label, every $\hat{A}$ letter that labels an edge of $s_2$ also labels an edge of $s_1$ (and thus labels an edge of $\Phi_N(\Pi)$).  Therefore every non-$\hat{A}$ letter that labels an edge of $s_2$ either appears in $\text{Lab}(p_1)$ or appears in an element of $\mathcal{T}(w)$.\\

Let $\Delta_1$ be the diagram obtained by deleting $\Phi_N(\Pi)$ and $\Pi$ from $\Delta$.  Then $\partial\Delta_1=p_2 q_2^{-1} s_2^{-1} q_1^{-1}$.  By the induction hypothesis, every $\hat{A}$ letter that labels an edge in $\Delta_1$ appears either in $\text{Lab}(\partial\Delta_1)$ or in $\mathcal{T}(\text{Lab}(\partial\Delta_1))$.  Since $\Delta$ contains no $\hat{\Theta}$ annuli, $\mathcal{T}(\text{Lab}(\partial\Delta_1)) \subseteq \mathcal{T}(w)$.  Therefore every $\hat{A}$ letter that appears in $\Delta_1$ appears either in $w$ or in an element of $\mathcal{T}(w)$.\\

Since every $\hat{A}$ letter that labels an edge of $\Delta$ also labels an edge of either $\Delta_1$ or $\Phi_N(\Pi)$, part 1 is proved.\\

To prove part 2, note that, since $\Phi_N(\Pi)$ contains no discs, every $\hat{Q}$ edge in $s_1$ is contained either in the boundary of $\Delta$ or in the boundary of a transition 2-cell of $\Delta$.  Since $|s_1|\geq \frac{1}{2}|\partial\Pi|$, by the definition of a disc label every $\hat{Q}$ letter that labels an edge of $\Pi$ labels an edge of $s_1$.  It follows from the induction hypothesis and the fact that $q_1,q_2$ are $\hat{\Theta}$ paths that every $\hat{Q}$ letter that labels an edge of $\Delta_1$ either appears in $p_1$, $s_2$, or in an element of $\mathcal{T}(\text{Lab}(\partial\Delta)\subseteq \mathcal{T}(w)$.  This proves part 2.\\

\end{proof}

\begin{lemma}\label{Psolv}

If the language accepted by $M_{\infty}$ is decidable, then the word problem is solvable for $P(M_{\infty})$.

\end{lemma}

\begin{proof}

Let $w$ be an arbitrary word in $P(M_{\infty})$.  By Lemma \ref{mindiagbounds}, if $\Delta$ is a minimal $P_D(M_{\infty})$ diagram for $w$, then the sum of the boundary lengths of the discs in $\Delta$ does not exceed $O(|w|^2)$ and the number of transition and auxiliary 2-cells in $\Delta$ does not exceed $O(|w|^3)$.  To prove Lemma \ref{Psolv}, it will be sufficient to show that we can effectively construct from $w$ a finite set of $P_D(M_{\infty})$ relators $\mathcal{R}(w)$ such that if $w$ is trivial in $P(M_{\infty})$ with minimal $P_D(M_{\infty})$ diagram $\Delta$, then the label of every 2-cell in $\Delta$ is contained in $\mathcal{R}(w)$.\\

By Lemma \ref{mathcalt} we can effectively construct from $w$ the finite set $\mathcal{T}(w)$ of transition relators.  Recall that if $w$ is trivial in $P(M_{\infty})$ with minimal $P_D(M_{\infty})$ diagram $\Delta$, then the boundary label of every transition 2-cell in $\Delta$ is in $\mathcal{T}(w)$.  It remains to construct analogous finite sets of auxiliary relators and disc labels.\\

We effectively construct $\mathcal{Q}(w)$, and $\mathcal{A}(w)$ from $w$, as in Lemma \ref{mathcala}.  If $w$ is trivial in $P(M_{\infty})$ with minimal $P_D(M_{\infty})$ diagram $\Delta$, then every auxiliary relator that labels a 2-cell of $\Delta$ is of the form $xyx^{-1}y^{-1}$, where $x$ is a $\hat{\Theta}$ letter that appears in an element of $\mathcal{T}(w)$ and $y$ is either an $\hat{A}$ letter in $\mathcal{A}(w)$ or a $\kappa$ letter.  Therefore we can effectively construct from $w$ a finite set $\mathcal{R}_1(w)$ of auxiliary relators of $P_D(M_{\infty})$ such that if $w$ is trivial in $P(M_{\infty})$ with minimal $P_D(M_{\infty})$ diagram $\Delta$, then the label of every auxiliary 2-cell in $\Delta$ is contained in $\mathcal{R}_1(w)$.\\

There are finitely many admissible words of $S(M'_{\infty})$ whose length does not exceed $O(|w|^2)$, whose $\hat{A}$ letters are elements of $\mathcal{A}(w)$, and whose $\hat{Q}$ letters are elements of $\mathcal{Q}(w)$.  We can effectively construct this finite set of admissible words.  For each such admissible word $W$, $K(W)$ is a disc label of $P(M_{\infty})$ if and only if $W$ is an acceptable admissible word of $S(M'_{\infty})$.  By Lemmas \ref{symsol} and \ref{smachsol}, it is decidable whether $W$ is an acceptable admissible word of $S(M'_{\infty})$.  Therefore, we can effectively construct from $w$ a finite set $\mathcal{R}_2(w)$ of disc labels of $P_D(M_{\infty})$ such that if $w$ is trivial in $P(M_{\infty})$ with minimal $P_D(M_{\infty})$ diagram $\Delta$, then the label of every disc in $\Delta$ is contained in $\mathcal{R}_2(w)$.\\

We set $\mathcal{R}(w)=\mathcal{T}(w)\cup \mathcal{R}_1(w)\cup \mathcal{R}_2(w)$ to complete the proof.\\

\end{proof}

\begin{lemma}\label{minimal1}

The presentation $P(M_{\infty})$ is a minimal presentation.

\end{lemma}

\begin{proof}

Consider any non-hub relator $\rho \textbf{x}_{\ell} \rho^{-1}\textbf{y}_{\ell}^{-1}$ of $P(M_{\infty})$, where $\rho\in \hat{\Theta}$.  Suppose we remove this relator from $P(M_{\infty})$ and then attempt to construct a minimal $P_D(M_{\infty})$ diagram $\Delta$ with boundary label $\rho \textbf{x}_{\ell} \rho^{-1}\textbf{y}_{\ell}^{-1}$ using 2-cells labeled by the remaining relators.  If $\Delta$ contains at least one hub, then there are at least $(4N-3)$ $\kappa$-edges in $\partial\Delta$, by Lemma \ref{kappahub}.  Therefore $\Delta$ contains no hubs.  By Lemma \ref{noThetaAn}, $\Delta$ contains no $\hat{\Theta}$ annuli.  Therefore $\Delta$ consists of a single $\rho$ band.  By the proof of Lemma \ref{hnn}, and the fact that no $\kappa$ letters appear in the relators of $S(M_{\infty})$, the words that appear on the tops of the $\rho$ relators of $P(M_{\infty})$ freely generate a subgroup of $F(\hat{Q}\cup \bar{Y})$.  Therefore $\Delta$ cannot exist.\\

If we remove the hub relator from $P(M_{\infty})$ then the only $P_D(M_{\infty})$ diagrams we can make using 2-cells labeled by the remaing relators contain no hubs.  Such diagrams contain no $\hat{\Theta}$ annuli by Lemma \ref{noThetaAn}.  Since disc labels contain no $\hat{\Theta}$ edges, it is impossible to construct a $P(M_{\infty})$ diagram whose boundary label is a disc label without using hub relators.  We conclude that $P(M_{\infty})$ is a minimal presentation.\\

\end{proof}

\section{Construction of $P'_1(M_{\infty})$ and Proof of Theorem \ref{main}}

For the purposes of this section, it would be convenient if every relator of $P(M_{\infty})$ were either a strictly positive word or a strictly negative word.  We begin by proving that $P(M_{\infty})$ can be transformed into a presentation $P_1(M{_\infty})$ in which every relator is either a strictly positive or a strictly negative word such that $P_1(M{_\infty})$ inherits all the desirable properties of $P(M_{\infty})$.\\

If $P(M_{\infty})=\langle X \| R\rangle$, then we construct $P_1(M{_\infty})$ by adding new generators and relators to $P_1(M{_\infty})$ as follows.  Let $\tilde{X}$ be a set of symbols that is in bijective correspondence with $X$.  For $g\in X$, let the corresponding element of $\tilde{X}$ be denoted $\tilde{g}$.  The generating set of $P_1(M{_\infty})$ is $Z= X\cup \tilde{X}$.  Let $\tilde{R}$ denote the set of relators obtained by replacing each negative letter $g^{-1}$ appearing in each relator of $R$ with the positive letter $\tilde{g}$.  The relator set of $P_1(M{_\infty})$ is obtained by taking the closure of $\tilde{R}\cup \{g \tilde{g}|g\in X\}$ under inverses and cyclic shifts.  This completes the construction of $P_1(M{_\infty})$, in which every relator is either strictly positive or strictly negative.  For every word $w$ in $P_1(M{_\infty})$, there exists a strictly positive word $w_p$ in $P_1(M{_\infty})$, which is obtained by replacing every negative letter $g^{-1}$ (or $\tilde{g}^{-1}$) in $w$ with the positive letter $\tilde{g}$ (or $g$).  We call $w_p$ the {\em strictly positive word in $P_1(M{_\infty})$ representing $w$}.  We state the following lemma without proof.

\begin{lemma}\label{obvious}

\begin{enumerate}

\item
The map $\iota$ given by $\iota(g)= g$, $\iota(\tilde{g})= g$ defines an isomorphism from $P(M_{\infty})$ to $P_1(M_{\infty})$.

\item
The word problem for $P(M_{\infty})$ is solvable if and only if the word problem for $P_1(M_{\infty})$ is solvable.

\item
The presentation $P(M_{\infty})$ is c.e. if and only if the presentation $P_1(M_{\infty})$ is c.e..

\item
The presentation $P(M_{\infty})$ is minimal if and only if the presentation $P_1(M_{\infty})$ is minimal.

\item
If $w_1$ is a trivial word in $P_1(M_{\infty})$ and $\Delta_1$ is a minimal area $P_1(M_{\infty})$ diagram for $w_1$ with area $n$, then a minimal area $P(M_{\infty})$ diagram $\Delta$ for $\iota(w_1)$ has area between $n$ and $n+|w_1|$.

\item
The Dehn fucntion for $P(M_{\infty})$ is equivalent to the Dehn function for  $P_1(M_{\infty})$.

\end{enumerate}

\end{lemma}

To prove Theorem \ref{main} we will use $P_1(M_{\infty})$ to construct a finitely generated group presentation $P_1'(M_{\infty})$ which will inherit desired properties from $P_1(M_{\infty})$.  For the rest of this section we will assume that $M_{\infty}$ is a c.e. union machine.  It follows from Lemmas \ref{presce} and \ref{obvious} that $P_1(M_{\infty})$ is c.e..  Therefore there exist c.e. sequences $E_1$ and $E_2$ such that $E_1$ is a sequence of positive generators of $P_1(M_{\infty})$ in which every positive generator of $P_1(M_{\infty})$ appears exactly once, and $E_2$ is a sequence of relators of $P_1(M_{\infty})$ in which every relator of $P_1(M_{\infty})$ appears exactly once.\\

We use $g_i$ to denote the $i$th term of the sequence $E_1$.\\

Let $\{b,a\}$ be a set of generating symbols that do not appear in the generating set $Z$ of $P_1(M)$.  These will be the generators of the finitely generated group presentation $P_1'(M_{\infty})$.  We define a map $h$ from $F(Z)$ to $F(b,a)$ as follows.

\begin{equation} \label{emb}
h(g_i)= a^{100}b^{i}a^{101}b^{i} \dots a^{199}b^{i}.
\end{equation}\\

Let $H$ be the subgroup of $F(a,b)$ generated by the words $\{h(g)|g\in Z\}$.  We define an {\em $H$-word} to be a reduced word in $F(a, b)$ representing an element of $H$.  A cyclic $H$-word is a cyclic word that is a cyclic conjugate of an $H$-word.  Note that when we refer to the $h$-image of a word $w$ in $P_1(M_{\infty})$, we mean the non-reduced word obtained by replacing each letter $g_i^{\pm1}$ in $w$ with $h(g_i)^{\pm1}$.  So the $h$-image of a word in $P_1(M_{\infty})$ is not necessarily an $H$-word.\\

Let $P'_1(M_{\infty})$ be the presentation whose generating set is $\{a,b\}$ and whose relator set is the set of $h$-images of relators of $P_1(M_{\infty})$.  \\

\begin{lemma}\label{finaldecidable}

For a c.e. union machine $M_{\infty}$, the presentation $P'_1(M_{\infty})$ is decidable.

\end{lemma}

\begin{proof}

We will first prove that if $M_{\infty}$ is c.e., then we can decide whether a given word $w$ in the generators of $P_1(M_{\infty})$ is an element of $\tilde{R}\cup \{g \tilde{g}|g\in X\}$, the relator set of $P_1(M_{\infty})$.  We can immediately decide whether $w\in \{g \tilde{g}|g\in X\}$.  Since $w\in \tilde{R}$ if and only if $\iota(w)$ (as defined in Lemma \ref{obvious}) is in the relator set $R$ of $P(M_{\infty})$, it will be sufficient to show that it is decidable whether $\iota(w)$ is in $R$.\\

Since there is only one hub relator in $R$, we can immediately decide whether $\iota(w)$ is the hub relator.  All other relators in $R$ contain command symbols.  Recall that the command symbols of $P(M_{\infty})$ contain a significant amount of information.  In particular, there is an algorithm that takes as input a command symbol $\rho$  of $P(M_{\infty})$ and outputs the finite set of transition relators of $P(M_{\infty})$ in which the letter $\rho$ appears (by Lemma \ref{comenu}).  We can use this algorithm to effectively determine from the command letters that appear in $\iota(w)$ whether or not $\iota(w)$ is a transition relator of $P(M_{\infty})$. \\

If $\iota(w)$ is a commutator of a command letter with a $\kappa$ or $\hat{A}$ letter, then $\iota(w)$ is an auxiliary relator of $P(M_{\infty})$.\\

If $\iota(w)$ is neither the hub relator, a transition relator, nor an auxiliary relator of $P(M_{\infty})$, then $\iota(w)$ is not a relator of $P(M_{\infty})$.\\

Note that this does not mean that $P_1(M_{\infty})$ is decidable.  In fact if $M_{\infty}$ is not c.e. and not decidable, then $P_1(M_{\infty})$ is not decidable because the generating set of $P_1(M_{\infty})$ is not decidable.  We have merely proven that it is possible to decide whether $w$, a given word in the generators of $P_1(M_{\infty})$, is a relator of $P_1(M_{\infty})$.\\

It is not necessary to decide the generating set of $P_1(M_{\infty})$ in order to decide $P'_1(M_{\infty})$. The generating set of $P'_1(M_{\infty})$ is finite. Thus, in order to prove that $P'_1(M_{\infty})$ is decidable, it will be sufficient to show that the relator set of $P'_1(M_{\infty})$ is decidable.  Suppose we are given a word $w'$ in the generators of $P'_1(M_{\infty})$.  It follows from the definition of $h$ that it is decidable whether $w'$ is an $H$-word.  If $w'$ is an $H$-word, then by examining the powers of $b$ that appear in $w'$ we can recover the tuple of indices $i_1,\dots,i_n$ such that $h(g_{i_1}\dots g_{i_n})=w'$.  We can then use these indices and the aforementioned c.e. sequence $E_1$ of generators to effectively construct the $P_1(M_{\infty})$ word $h^{-1}(w')=g_{i_1}\dots g_{i_n}$.  The word $w'$ is a relator of $P'_1(M_{\infty})$ if and only if $h^{-1}(w')$ is a relator of $P_1(M_{\infty})$, which is decidable by the above argument.  \\

\end{proof}

\begin{lemma}\label{citedstuff}

The map $h$ is an embedding of $P_1(M_{\infty})$ into $P_1'(M_{\infty})$.  The word problem is solvable for $P_1(M_{\infty})$ if and only if the word problem is solvable for $P_1'(M_{\infty})$.

\end{lemma}

\begin{proof}

By  \cite[Lemma 8]{FG} the map $h$ is an embedding of $P_1(M_{\infty})$ into $P_1'(M_{\infty})$.  By \cite[Lemma 11]{FG} if the word problem is solvable for $P_1(M_{\infty})$, then the word problem is solvable for $P_1'(M_{\infty})$.  If the word problem is solvable for $P_1'(M_{\infty})$, then it follows from the fact that $h$ is an embedding that the word problem is solvable for $P_1(M_{\infty})$.

\end{proof}

\begin{lemma}\label{finalminimal}

$P_1'(M_{\infty})$ is a minimal presentation.

\end{lemma}

\begin{proof}

By Lemmas \ref{minimal1} and \ref{obvious}, $P_1(M_{\infty})$ is a minimal presentation.  Let $r'$ be a relator of $P_1'(M_{\infty})$ and let $r=h^{-1}(r')$ be the corresponding relator of $P_1(M_{\infty})$.  Let $P_1(M_{\infty})\setminus \{r\}$ and $P_1'(M_{\infty})\setminus \{r'\}$ denote the presentations obtained by removing $r$ and $r'$ from $P_1(M_{\infty})$ and $P_1'(M_{\infty})$, respectively.  Since $P_1(M_{\infty})$ is minimal, $r$ is not trivial in $P_1(M_{\infty})\setminus \{r\}$.  By Lemma \ref{citedstuff}, $h$ is an embedding of $P_1(M_{\infty})\setminus \{r\}$  into $P_1'(M_{\infty})\setminus \{r'\}$ and therefore $r'$ is not trivial in $P_1'(M_{\infty})\setminus \{r'\}$.  This proves the lemma.

\end{proof}

We denote the Dehn function for $P_1(M_{\infty})$ by $f$ and the Dehn function for $P_1'(M_{\infty})$ by $f'$

\begin{lemma}\label{Dehnfmore}

$f' \preceq f$.

\end{lemma}

\begin{proof}

If $w'$ is an $H$-word and $w'$ is trivial in $P_1'(M_{\infty})$, then by the definition of $h$, $|h^{-1}(w')|\leq |w'|$.  Since $h$ is an embedding, there is a $P_1(M_{\infty})$ diagram $\Delta$ with boundary $h^{-1}(w')$.  The area of $\Delta$ is at most $f(|h^{-1}(w')|)\leq f(|w'|)$.  If we replace each edge $e$ in $\Delta$ with a path $p_e$ such that $\text{Lab}(p_e)=h(\text{Lab}(e))$, then the resulting object is a $P_1'(M_{\infty})$ diagram with area not exceeding $f(|w'|)$ and whose boundary label is freely equal to $w'$.\\

If $w'$ is a trivial word in $P_1'(M_{\infty})$ that is not an $H$-word, then it follows from part 4 of \cite[Lemma 8]{FG} that $w'$ is a product of conjugates of $H$-words $w'_1,\dots w'_n$, such that $\Sigma_{i=1}^n |w'_i|\leq O(w')$.  By the above paragraph and the fact that $f$ is equivalent to a superadditive function, there is a $P_1'(M_{\infty})$ diagram for $w'$ with area not exceeding $f(O(w'))$.

\end{proof}

In order to prove Theorem \ref{main}, we must prove that $f'\succeq f$.  If $\Delta'$ is a $P_1'(M_{\infty})$ diagram, we define a path $p$ in $\Delta'$ to be an {\em $h$-path} if the label of $p$ is $h(x)$ for some $x\in Z$ and either $p$ is a subpath of the boundary of a 2-cell of $\Delta'$ or no edges of $p$ are contained in the boundary of a 2-cell of $\Delta'$.\\

Suppose $\Delta'$ is a $P_1'(M_{\infty})$ diagram containing distinct $h$-paths $p_1$ and $p_2$.  If $p_1$ and $p_2^{-1}$ share a common edge or a common vertex that is not an endpoint of both $p_1$ and $p_2$, then $p_1$ and $p_2$ are \textit{adjacent} $h$-paths.  If $p_1$ and $p_2^{-1}$ share a common edge, then $p_1$ and $p_2$ are {\em edge adjacent} $h$-paths.  Suppose $p_1$ and $p_2$ are adjacent $h$-paths where $q$ is a common subpath of $p_1$ and $p_2^{-1}$ such that $|q|>0$, $p_1=u_1q\beta_1$, and $p_2^{-1}=u_2q\beta_2$.  If $\text{Lab} (u_1) =\text{Lab} (u_2)$ and $\text{Lab} (\beta_1)= \text{Lab} (\beta_2)$, then we call $p_1$ and $p_2$ \textit{strongly adjacent} $h$-paths.  If $p_1=p_2^{-1}$, then $p_1$ and $p_2$ are {\em contiguous} $h$-paths. \\

\begin{lemma}\label{fix1}

Suppose that $\Delta'$ is a $P_1'(M_{\infty})$ diagram in which every edge is contained in an $h$-path and no two non-contiguous $h$-paths are adjacent.  Then there is a $P_1(M_{\infty})$ diagram $\Delta$ with boundary label $w$ such that $\text{Lab}(\partial\Delta')=h(w)$ and the area of $\Delta$ is the same as that of $\Delta'$.

\end{lemma}

\begin{proof}

Since every edge is contained in an $h$-path and no two non-contiguous $h$-paths are adjacent, we can replace each $h$-path $p$ of $\Delta'$ with a single edge $e$ such that $h(\text{Lab}(e))=\text{Lab}(p)$ and call the resulting $P_1(M_{\infty})$ diagram $\Delta$.  It follows that $\Delta$ has boundary label $w$ such that $\text{Lab}(\partial\Delta')=h(w)$ and the area of $\Delta$ is the same as that of $\Delta'$.

\end{proof}

\begin{lemma}\label{fix3}

If $\Delta'$ is a minimal area $P_1'(M_{\infty})$ diagram whose boundary label $w'$ is the $h$-image of a strictly positive trivial word in $P_1(M_{\infty})$, then there exists a $P_1'(M_{\infty})$ diagram with boundary label $w'$ and the same area as $\Delta'$ in which every edge is contained in an $h$-path and no two $h$-paths are adjacent.

\end{lemma}

\begin{proof}

We first create a spherical diagram $\hat{\Delta}'$ by gluing a single 2-cell $\pi_0$ with boundary label $w'^{-1}$ to the boundary of $\Delta'$.  We consider subpaths of $\partial\pi_0$ whose labels are $h$-images of elements of $Z$ to be $h$-paths even though $\pi_0$ is not a $P'_1(M_{\infty})$ 2-cell.  We first show that $\hat{\Delta}'$ can be transformed via folding surgeries into a spherical diagram in which every pair of strongly adjacent $h$-paths are contiguous.\\

Suppose that $p_1,p_2$ are strongly adjacent $h$-paths in $\hat{\Delta}'$.  Suppose $p_1=\mu_1q\beta_1$ and $p_2^{-1}= \mu_2 q \beta_2$, where $q$ is a common subpath of $p_1$ and $p_2^{-1}$ such that $\text{Lab}(\mu_1)=\text{Lab}(\mu_2)$ and $\text{Lab}(\beta_1)=\text{Lab}(\beta_2)$.\\

We will identify $p_1$ and $p_2^{-1}$ edge by edge by performing folding surgeries.  We start by performing a folding surgery to identify the final edges of $\mu_1$ and $\mu_2$.  We denote these edges by $e_1$ and $e_2$ respectively.  Since $\text{Lab}(\mu_1)=\text{Lab}(\mu_2)$, we know that $e_1$ and $e_2$ have the same label.  Since $\mu_1$ and $\mu_2$ share the same final vertex, $e_1$ and $e_2$ share the same final vertex.  Since $\Delta'$ is a minimal area diagram, $e_1$ and $e_2$ do not share the same initial vertex.  We can thus perform a folding surgery at the path $e_1e_2^{-1}$ to increase by one the number of edges shared by $p_1$ and $p_2^{-1}$, as shown in the below figure.  \\

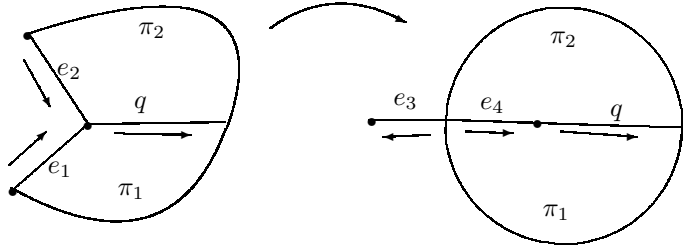
\begin{center}

\unitlength 1mm 
\linethickness{0.4pt}
\ifx\plotpoint\undefined\newsavebox{\plotpoint}\fi 
\begin{picture}(103.984,52.25)(0,0)
\qbezier(17,41)(51.75,52.25)(43.5,28.5)
\qbezier(43.25,28.25)(36,8.625)(14.75,20.5)
\multiput(14.75,20)(.0383895131,.0337078652){267}{\line(1,0){.0383895131}}
\multiput(25,29)(-.033695652,.051086957){230}{\line(0,1){.051086957}}
\multiput(25,29)(2.28125,.03125){8}{\line(1,0){2.28125}}
\put(30.75,20){\makebox(0,0)[cc]{$\pi_1$}}
\put(33.5,41.5){\makebox(0,0)[cc]{$\pi_2$}}
\put(32,31.5){\makebox(0,0)[cc]{$q$}}
\put(22.5,36){\makebox(0,0)[cc]{$e_2$}}
\put(19.5,28){\vector(1,1){.07}}\multiput(14.5,23.5)(.037313433,.03358209){134}{\line(1,0){.037313433}}
\put(20,31.25){\vector(1,-2){.07}}\multiput(16.5,37.5)(.033653846,-.060096154){104}{\line(0,-1){.060096154}}
\put(38.75,27.5){\vector(1,0){.07}}\multiput(28.5,27.75)(1.28125,-.03125){8}{\line(1,0){1.28125}}
\put(103.984,28.75){\line(0,1){.7753}}
\put(103.965,29.525){\line(0,1){.7734}}
\put(103.908,30.299){\line(0,1){.7696}}
\multiput(103.812,31.068)(-.03328,.19099){4}{\line(0,1){.19099}}
\multiput(103.679,31.832)(-.028435,.126082){6}{\line(0,1){.126082}}
\multiput(103.509,32.589)(-.029668,.106738){7}{\line(0,1){.106738}}
\multiput(103.301,33.336)(-.03053,.092003){8}{\line(0,1){.092003}}
\multiput(103.057,34.072)(-.031134,.080344){9}{\line(0,1){.080344}}
\multiput(102.777,34.795)(-.031549,.070841){10}{\line(0,1){.070841}}
\multiput(102.461,35.503)(-.03182,.06291){11}{\line(0,1){.06291}}
\multiput(102.111,36.195)(-.031974,.05616){12}{\line(0,1){.05616}}
\multiput(101.727,36.869)(-.0320329,.0503228){13}{\line(0,1){.0503228}}
\multiput(101.311,37.524)(-.0320111,.045206){14}{\line(0,1){.045206}}
\multiput(100.863,38.156)(-.0319196,.0406689){15}{\line(0,1){.0406689}}
\multiput(100.384,38.766)(-.0317669,.0366063){16}{\line(0,1){.0366063}}
\multiput(99.876,39.352)(-.033532,.0349967){16}{\line(0,1){.0349967}}
\multiput(99.339,39.912)(-.0352156,.033302){16}{\line(-1,0){.0352156}}
\multiput(98.776,40.445)(-.039268,.0336281){15}{\line(-1,0){.039268}}
\multiput(98.187,40.949)(-.0408772,.0316524){15}{\line(-1,0){.0408772}}
\multiput(97.574,41.424)(-.0454148,.0317142){14}{\line(-1,0){.0454148}}
\multiput(96.938,41.868)(-.0505317,.0317024){13}{\line(-1,0){.0505317}}
\multiput(96.281,42.28)(-.056368,.031605){12}{\line(-1,0){.056368}}
\multiput(95.604,42.66)(-.063117,.031407){11}{\line(-1,0){.063117}}
\multiput(94.91,43.005)(-.071046,.031085){10}{\line(-1,0){.071046}}
\multiput(94.2,43.316)(-.080546,.030607){9}{\line(-1,0){.080546}}
\multiput(93.475,43.591)(-.092201,.029926){8}{\line(-1,0){.092201}}
\multiput(92.737,43.831)(-.10693,.028968){7}{\line(-1,0){.10693}}
\multiput(91.989,44.033)(-.151518,.033129){5}{\line(-1,0){.151518}}
\multiput(91.231,44.199)(-.19121,.03203){4}{\line(-1,0){.19121}}
\put(90.466,44.327){\line(-1,0){.7702}}
\put(89.696,44.418){\line(-1,0){.7737}}
\put(88.922,44.47){\line(-1,0){.7754}}
\put(88.147,44.484){\line(-1,0){.7751}}
\put(87.372,44.46){\line(-1,0){.773}}
\put(86.599,44.397){\line(-1,0){.769}}
\multiput(85.83,44.297)(-.152618,-.027626){5}{\line(-1,0){.152618}}
\multiput(85.067,44.159)(-.125893,-.02926){6}{\line(-1,0){.125893}}
\multiput(84.311,43.983)(-.106541,-.030367){7}{\line(-1,0){.106541}}
\multiput(83.566,43.771)(-.091801,-.031132){8}{\line(-1,0){.091801}}
\multiput(82.831,43.522)(-.080138,-.03166){9}{\line(-1,0){.080138}}
\multiput(82.11,43.237)(-.070633,-.032013){10}{\line(-1,0){.070633}}
\multiput(81.404,42.916)(-.0627,-.032231){11}{\line(-1,0){.0627}}
\multiput(80.714,42.562)(-.055949,-.032341){12}{\line(-1,0){.055949}}
\multiput(80.043,42.174)(-.0501118,-.0323619){13}{\line(-1,0){.0501118}}
\multiput(79.391,41.753)(-.0449953,-.0323066){14}{\line(-1,0){.0449953}}
\multiput(78.761,41.301)(-.0404589,-.0321854){15}{\line(-1,0){.0404589}}
\multiput(78.154,40.818)(-.0363974,-.0320061){16}{\line(-1,0){.0363974}}
\multiput(77.572,40.306)(-.0327305,-.0317747){17}{\line(-1,0){.0327305}}
\multiput(77.015,39.766)(-.0330705,-.0354331){16}{\line(0,-1){.0354331}}
\multiput(76.486,39.199)(-.0333701,-.0394875){15}{\line(0,-1){.0394875}}
\multiput(75.986,38.607)(-.0336256,-.0440183){14}{\line(0,-1){.0440183}}
\multiput(75.515,37.99)(-.0314159,-.0456216){14}{\line(0,-1){.0456216}}
\multiput(75.075,37.352)(-.0313706,-.0507383){13}{\line(0,-1){.0507383}}
\multiput(74.667,36.692)(-.031235,-.056574){12}{\line(0,-1){.056574}}
\multiput(74.293,36.013)(-.030992,-.063321){11}{\line(0,-1){.063321}}
\multiput(73.952,35.317)(-.030618,-.071249){10}{\line(0,-1){.071249}}
\multiput(73.645,34.604)(-.030078,-.080745){9}{\line(0,-1){.080745}}
\multiput(73.375,33.877)(-.03351,-.105595){7}{\line(0,-1){.105595}}
\multiput(73.14,33.138)(-.032978,-.124971){6}{\line(0,-1){.124971}}
\multiput(72.942,32.388)(-.032136,-.151732){5}{\line(0,-1){.151732}}
\multiput(72.782,31.63)(-.03078,-.19141){4}{\line(0,-1){.19141}}
\put(72.659,30.864){\line(0,-1){.7708}}
\put(72.573,30.093){\line(0,-1){1.5495}}
\put(72.517,28.544){\line(0,-1){.7749}}
\put(72.547,27.769){\line(0,-1){.7726}}
\put(72.614,26.996){\line(0,-1){.7683}}
\multiput(72.719,26.228)(.028626,-.152434){5}{\line(0,-1){.152434}}
\multiput(72.862,25.466)(.030085,-.125698){6}{\line(0,-1){.125698}}
\multiput(73.043,24.712)(.031064,-.10634){7}{\line(0,-1){.10634}}
\multiput(73.26,23.967)(.031733,-.091595){8}{\line(0,-1){.091595}}
\multiput(73.514,23.235)(.032184,-.079929){9}{\line(0,-1){.079929}}
\multiput(73.804,22.515)(.032475,-.070422){10}{\line(0,-1){.070422}}
\multiput(74.129,21.811)(.032641,-.062487){11}{\line(0,-1){.062487}}
\multiput(74.488,21.124)(.032707,-.055736){12}{\line(0,-1){.055736}}
\multiput(74.88,20.455)(.0326896,-.0498987){13}{\line(0,-1){.0498987}}
\multiput(75.305,19.806)(.0326008,-.0447826){14}{\line(0,-1){.0447826}}
\multiput(75.762,19.179)(.0324499,-.0402471){15}{\line(0,-1){.0402471}}
\multiput(76.248,18.575)(.0322439,-.0361869){16}{\line(0,-1){.0361869}}
\multiput(76.764,17.996)(.0319885,-.0325216){17}{\line(0,-1){.0325216}}
\multiput(77.308,17.444)(.035649,-.0328376){16}{\line(1,0){.035649}}
\multiput(77.878,16.918)(.0397053,-.0331106){15}{\line(1,0){.0397053}}
\multiput(78.474,16.421)(.0442377,-.0333364){14}{\line(1,0){.0442377}}
\multiput(79.093,15.955)(.0493516,-.0335098){13}{\line(1,0){.0493516}}
\multiput(79.735,15.519)(.055188,-.033624){12}{\line(1,0){.055188}}
\multiput(80.397,15.116)(.061939,-.03367){11}{\line(1,0){.061939}}
\multiput(81.078,14.745)(.069875,-.033635){10}{\line(1,0){.069875}}
\multiput(81.777,14.409)(.079386,-.033501){9}{\line(1,0){.079386}}
\multiput(82.492,14.107)(.091058,-.033242){8}{\line(1,0){.091058}}
\multiput(83.22,13.841)(.105812,-.032817){7}{\line(1,0){.105812}}
\multiput(83.961,13.612)(.125184,-.032158){6}{\line(1,0){.125184}}
\multiput(84.712,13.419)(.15194,-.031141){5}{\line(1,0){.15194}}
\put(85.472,13.263){\line(1,0){.7664}}
\put(86.238,13.145){\line(1,0){.7713}}
\put(87.009,13.065){\line(1,0){.7743}}
\put(87.784,13.023){\line(1,0){.7755}}
\put(88.559,13.019){\line(1,0){.7747}}
\put(89.334,13.053){\line(1,0){.7721}}
\put(90.106,13.126){\line(1,0){.7676}}
\multiput(90.874,13.236)(.152243,.029624){5}{\line(1,0){.152243}}
\multiput(91.635,13.384)(.125498,.030908){6}{\line(1,0){.125498}}
\multiput(92.388,13.57)(.106134,.03176){7}{\line(1,0){.106134}}
\multiput(93.131,13.792)(.091385,.032332){8}{\line(1,0){.091385}}
\multiput(93.862,14.051)(.079717,.032707){9}{\line(1,0){.079717}}
\multiput(94.579,14.345)(.070207,.032936){10}{\line(1,0){.070207}}
\multiput(95.281,14.674)(.062272,.03305){11}{\line(1,0){.062272}}
\multiput(95.966,15.038)(.055521,.033072){12}{\line(1,0){.055521}}
\multiput(96.633,15.435)(.0496834,.0330159){13}{\line(1,0){.0496834}}
\multiput(97.279,15.864)(.044568,.0328935){14}{\line(1,0){.044568}}
\multiput(97.903,16.325)(.0400336,.0327129){15}{\line(1,0){.0400336}}
\multiput(98.503,16.815)(.0359748,.0324804){16}{\line(1,0){.0359748}}
\multiput(99.079,17.335)(.0323113,.0322009){17}{\line(1,0){.0323113}}
\multiput(99.628,17.882)(.0326033,.0358635){16}{\line(0,1){.0358635}}
\multiput(100.15,18.456)(.0328497,.0399214){15}{\line(0,1){.0399214}}
\multiput(100.642,19.055)(.0330458,.0444552){14}{\line(0,1){.0444552}}
\multiput(101.105,19.677)(.0331857,.0495702){13}{\line(0,1){.0495702}}
\multiput(101.536,20.322)(.033262,.055407){12}{\line(0,1){.055407}}
\multiput(101.936,20.987)(.033263,.062159){11}{\line(0,1){.062159}}
\multiput(102.301,21.67)(.033176,.070094){10}{\line(0,1){.070094}}
\multiput(102.633,22.371)(.03298,.079604){9}{\line(0,1){.079604}}
\multiput(102.93,23.088)(.032645,.091274){8}{\line(0,1){.091274}}
\multiput(103.191,23.818)(.032123,.106025){7}{\line(0,1){.106025}}
\multiput(103.416,24.56)(.031337,.125392){6}{\line(0,1){.125392}}
\multiput(103.604,25.313)(.030145,.15214){5}{\line(0,1){.15214}}
\put(103.755,26.073){\line(0,1){.7672}}
\put(103.868,26.84){\line(0,1){.7718}}
\put(103.943,27.612){\line(0,1){1.1377}}
\multiput(72.75,29.5)(1.0333333,-.0333333){30}{\line(1,0){1.0333333}}
\put(72.5,29.5){\line(-1,0){10.25}}
\put(17,40.75){\circle*{1.118}}
\put(25,28.75){\circle*{1.118}}
\put(15,20){\circle*{1.118}}
\put(62.75,29.25){\circle*{1.118}}
\put(84.75,29){\circle*{1.118}}
\put(67.25,32){\makebox(0,0)[cc]{$e_3$}}
\put(78.75,31.5){\makebox(0,0)[cc]{$e_4$}}
\put(95.25,30.5){\makebox(0,0)[cc]{$q$}}
\put(88.25,40.25){\makebox(0,0)[cc]{$\pi_2$}}
\put(87.25,17.25){\makebox(0,0)[cc]{$\pi_1$}}
\put(64.25,27.25){\vector(-1,0){.07}}\multiput(70.5,27.5)(-.78125,-.03125){8}{\line(-1,0){.78125}}
\put(81.5,27.75){\vector(1,0){.07}}\multiput(75.25,28)(.78125,-.03125){8}{\line(1,0){.78125}}
\put(98,27.25){\vector(1,0){.07}}\multiput(87.75,28)(.4456522,-.0326087){23}{\line(1,0){.4456522}}
\put(67.25,42.5){\vector(2,-1){.07}}\qbezier(49.25,41.75)(57.75,47.875)(67.25,42.5)
\put(21.25,23){\makebox(0,0)[cc]{$e_1$}}
\end{picture}

\end{center}

We repeat this process until $p_1$ and $p_2^{-1}$ have been identified.  We claim that this process of identifying $p_1$ and $p_2^{-1}$ decreases the number of $h$-paths in $\hat{\Delta}'$ that are not contained in a contiguous pair of $h$-paths.  Consider the set $\Psi$ of $h$-paths that were edge adjacent to $p_1$ or $p_2^{-1}$ before any folding surgeries took place.  Note that before the folding surgeries are performed, no $h$-path in $\Psi$ is contiguous to another $h$-path in $\hat{\Delta}'$.  Also note that if an $h$-path $p_3$ of $\hat{\Delta}'$ is not contained in $\Psi$, then the folding surgeries performed to identify $p_1$ and $p_2^{-1}$ have no effect on which $h$-paths are edge adjacent to $p_3$ .  Therefore no pairs of contiguous $h$-paths become non-contiguous as a result of these folding surgeries, and the process of identifying $p_1$ and $p_2^{-1}$ decreases the number of $h$-paths of $\hat{\Delta}'$ that are not contained in a contiguous pair of $h$-paths by at least two.\\

Since there are only finitely many $h$-paths in $\hat{\Delta}'$, this process of transforming pairs of non-contiguous strongly adjacent $h$-paths into pairs of contiguous $h$-paths must terminate after some finite number of identifications.  At this point $\hat{\Delta}'$ will contain no more pairs of non-contiguous strongly adjacent $h$-paths.\\

We define an equivalence relation $\equiv_e$ on the set of $h$-paths of $\hat{\Delta}'$: if $p_{1},p_{2}$ are $h$-paths in $\hat{\Delta}'$, then $p_{1}\equiv_e p_{2}$ if there is a sequence of $h$-paths $p_{i_1}\dots p_{i_n}$ such that $p_{i_1}=p_1$, $p_{i_n}=p_2$, and $p_{i_j}$ is edge adjacent to $p_{i_{j+1}}$.\\

We now prove that if $\hat{\Delta}'$ contains no pairs of non-contiguous strongly adjacent $h$-paths then $\hat{\Delta}'$ contains no pairs of non-contiguous edge adjacent $h$-paths.  Suppose towards contradiction that $\hat{\Delta}'$ contains a pair of non-contiguous edge adjacent $h$-paths.  Then the there is a $\equiv_e$ equivalence class $\mathcal{E}_1$ that contains more than 2 elements.  Note that $\mathcal{E}_1$ contains no pairs of contiguous $h$-paths and therefore no pairs of strongly adjacent $h$-paths.\\

We draw an undirected graph $G$ on the spherical diagram $\hat{\Delta}'$ as follows.  For each 2-cell $\pi$ of $\hat{\Delta}'$, we place a vertex $v_{\pi}$ in the interior of $\pi$.  Let $p^{\pi}_1,\dots p^{\pi}_n$ be the $h$-paths in $\partial\pi$ and let $v^{\pi}_i$ be the initial vertex of $p^{\pi}_i$ in $\partial{\pi}$.  Set $V_{\pi}=\{v_{\pi},v^{\pi}_1\dots v^{\pi}_n\}$.  The vertices of $G$ are $\cup_{\pi\in\hat{\Delta}'} V_{\pi}$.  We draw non-intersecting undirected edges in the interior of each $\pi$ connecting $v_\pi$ with each $v^{\pi}_i$.  These are the edges of $G$.

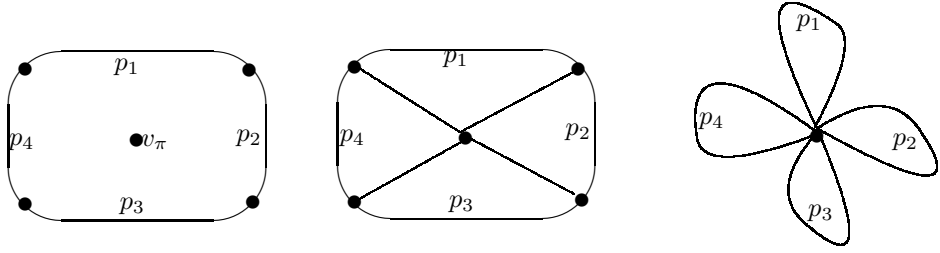
\begin{center}

\unitlength 1mm 
\linethickness{0.4pt}
\ifx\plotpoint\undefined\newsavebox{\plotpoint}\fi 
\begin{picture}(131.25,73)(0,0)
\put(18.625,48.5){\oval(34.25,22.5)[]}
\put(62.375,48.75){\oval(34.25,22.5)[]}
\put(3.75,57.5){\circle*{1.803}}
\put(47.5,57.75){\circle*{1.803}}
\put(33.5,57.25){\circle*{1.803}}
\put(77.25,57.5){\circle*{1.803}}
\put(34,39.75){\circle*{1.803}}
\put(77.75,40){\circle*{1.803}}
\put(3.75,39.5){\circle*{1.803}}
\put(47.5,39.75){\circle*{1.803}}
\put(18.5,48){\circle*{1.803}}
\put(62.25,48.25){\circle*{1.803}}
\put(109,48.5){\circle*{1.803}}
\multiput(47.75,57.75)(.0528846154,-.0336538462){260}{\line(1,0){.0528846154}}
\multiput(61.5,49)(.0634920635,.0337301587){252}{\line(1,0){.0634920635}}
\multiput(62.25,48.5)(.063043478,-.033695652){230}{\line(1,0){.063043478}}
\multiput(62.25,48)(-.0602040816,-.0336734694){245}{\line(-1,0){.0602040816}}
\qbezier(109,49.5)(97.75,73)(111.5,63.5)
\qbezier(111.5,63.5)(114.875,61.375)(108.75,49.75)
\qbezier(109,49.5)(131.25,37)(122.5,49.5)
\qbezier(122.5,49.5)(118.25,55.5)(109,49.5)
\qbezier(109,49.5)(118.375,27.25)(107.25,37)
\qbezier(107.25,37)(102.375,41.5)(109,49)
\qbezier(109,49)(91,62.25)(93,48.5)
\qbezier(93,48.5)(93.5,42.875)(109,48.75)
\put(17.25,57.75){\makebox(0,0)[cc]{$p_1$}}
\put(33.5,48.25){\makebox(0,0)[cc]{$p_2$}}
\put(18,39.25){\makebox(0,0)[cc]{$p_3$}}
\put(3.25,47.75){\makebox(0,0)[cc]{$p_4$}}
\put(47.25,48.25){\makebox(0,0)[cc]{$p_4$}}
\put(61,58.75){\makebox(0,0)[cc]{$p_1$}}
\put(77.25,49){\makebox(0,0)[cc]{$p_2$}}
\put(61.75,39.5){\makebox(0,0)[cc]{$p_3$}}
\put(108,63.5){\makebox(0,0)[cc]{$p_1$}}
\put(120.75,47.75){\makebox(0,0)[cc]{$p_2$}}
\put(109.5,38.25){\makebox(0,0)[cc]{$p_3$}}
\put(95,50.5){\makebox(0,0)[cc]{$p_4$}}
\put(21,47.75){\makebox(0,0)[cc]{$v_{\pi}$}}
\end{picture}

\end{center}

We regard $G$ as a subspace of $\hat{\Delta}'$ and contract $G$ to a point.  This contraction of $G$ transforms $\hat{\Delta}'$ into a 2-complex $\Upsilon$ that consists of finitely many spherical subcomplexes $\Upsilon_1,\dots \Upsilon_n$ such that for $i\neq j$, $\Upsilon_i\cup\Upsilon_j$ is either a single point or the empty set.  Note that each $\Upsilon_i$ is a spherical diagram over the presentation $\langle a,b \| h(g), g\in Z \rangle$; a presentation which satisfies the small cancelation condition $C'(\frac{1}{10})$.  We define the boundaries of the 2-cells of $\Upsilon$ to be {\em $h$-paths of $\Upsilon$}.\\

There is a natural bijection between the $h$-paths of $\hat{\Delta}'$ and those of $\Upsilon$.  If an $h$-path of $\hat{\Delta}'$ corresponds under this bijection with an $h$-path of $\Upsilon$, then we refer to these two $h$-paths interchangeably.  Since the contraction of $G$ does not affect edge adjacency of pairs of $h$-paths, if a pair of $h$-paths are edge adjacent in $\hat{\Delta}'$, then those two $h$-paths are edge adjacent in $\Upsilon$ and therefore contained in the same $\Upsilon_i$.  Thus for each $\Upsilon_i$ there is a corresponding $\equiv_e$ equivalence class $\mathcal{E}_i$ such the $h$-paths that form the boundaries of the 2-cells in $\Upsilon_i$ are exactly the $h$-paths in $\mathcal{E}_i$.  Let $\Upsilon_1$ be the subcomplex that contains the $h$-paths of the aforementioned equivalence class $\mathcal{E}_1$.\\

Since no $h$-paths in $\mathcal{E}_1$ are strongly adjacent, $\Upsilon_1$ does not contain a reducible pair of cells.  This is a contradiction, since it follows from the well-known Greendlinger's Lemma \cite{LS} that a spherical diagram over a $C'(\frac{1}{10})$ presentation must contain a reducible pair of 2-cells.\\

We conclude that after performing the folding surgeries described above, the diagram $\hat{\Delta}'$ contains no pairs of non-contiguous edge adjacent $h$-paths.  Since every $h$-path in a spherical $P'_1(M_{\infty})$ diagram is edge adjacent to another $h$-path, it follows that every $h$-path of $\hat{\Delta}'$ is contiguous to another $h$-path of $\hat{\Delta}'$.  Also, since $\hat{\Delta}'$ is spherical, every edge of $\hat{\Delta}'$ is contained in an $h$-path. If we now delete the 2-cell $\pi_0$ from $\hat{\Delta}'$, the resulting object is a $P_1'(M_{\infty})$ diagram with boundary label $w'$ in which every edge is contained in an $h$-path and no two non-contiguous $h$-paths are adjacent.  Additionally, this diagram has the same area as $\Delta'$.  This proves the lemma.

\end{proof}

\begin{lemma}\label{Dehnfless}

$f\preceq f'.$

\end{lemma}

\begin{proof}

If $\textbf{u}$ is an input word of $M_{\infty}$, then we call $\mathcal{K}(\textbf{u})$ an {\em input disc label}.  Since the input alphabet for $M_{\infty}$ is finite, the set of input configurations of $M_{\infty}$ is a language over a finite alphabet.  Therefore the set of input disc labels in $P_1(M_{\infty})$ is a language over a finite alphabet.   Therefore the lengths of these input disc labels (and the lengths of their strictly positive representations in $P_1(M_{\infty})$) differ from the lengths of their $h$-images by at most a constant factor.\\

Let $n>0$.  In the proof of \cite[Lemma 12.1]{BRS}, the authors show that there exists an acceptable input configuration $\textbf{u}$ of $M_{\infty}$ such that $\textbf{u}\leq n$ and the area of the minimal area $P(M_{\infty})$ diagram for $\mathcal{K}(\textbf{u})$ is greater than $O(T^4(n))$.  Let $\mathcal{K}(\textbf{u})^+$ be the strictly positive word in $P_1(M_{\infty})$ representing the element $\mathcal{K}(\textbf{u})$.  By Lemma \ref{obvious}, the area of the minimal area $P_1(M_{\infty})$ diagram for $\mathcal{K}(\textbf{u})^+$ is greater than $O(T^4(n))$.  Therefore, by Lemmas  \ref{fix1} and \ref{fix3}, the minimal area $P'_1(M_{\infty})$ diagram for $h(\mathcal{K}(\textbf{u})^+)$ has area greater than $O(T^4(n))$.  By the definitions of $\mathcal{K}$ and $h$ and the observation in the previous paragraph, $|h(\mathcal{K}(\textbf{u})^+)|=O(n)$.  Therefore $T^4\preceq f'$.  Since $f\equiv T^4$, this implies that $f\preceq f'$.\\

\end{proof}

We can now prove Theorem \ref{main}.  Let $M_{\infty}$ be as in the statement of Theorem \ref{main}.  We first observe that $h\circ \mathcal{K}:A^*\rightarrow \{a^{\pm1},b^{\pm1}\}^*$ is an injective map.  The presentation $P'_1(M_{\infty})$ is finitely generated and, by Lemma \ref{finaldecidable}, decidable.  By Lemma \ref{finalminimal}, $P'_1(M_{\infty})$ is minimal.  Part 1 of Theorem \ref{main} follows from Lemmas \ref{smalldeal} and \ref{obvious}. Part 2 follows from Lemmas \ref{Psolv} and \ref{obvious}.  Part 3 follows from Theorem \ref{smalldeal} and Lemmas \ref{obvious}, \ref{Dehnfless}, and \ref{Dehnfmore}. As for part 4, suppose $\textbf{u}\in L^*$ and that a minimal length accepting $M_{\infty}$ computation for $\textbf{u}$ has length $\ell(\textbf{u})$.  Then by Lemma \ref{smalldeal} part 3, the  minimal area $P(M_{\infty})$ diagram for $\mathcal{K}(\textbf{u})$ has area equal to $O(\ell(\textbf{u})^4)$.  Suppose that $w_1$ is the strictly positive word representing $\mathcal{K}(\textbf{u})$ in $P_1(M_{\infty})$.  By Lemma \ref{obvious} the minimal area $P_1(M_{\infty})$ diagram for $w_1$ has area equal to $O(\ell(\textbf{u})^4)$.  By Lemmas \ref{fix1} and \ref{fix3}, and the fact that $|h(w_1)|=O(|w_1|)$ (since $\mathcal{K}(\textbf{u})$ is an input disc label), the minimal area $P'_1(M_{\infty})$ diagram for $h(w_1)$ has area $O(\ell(\textbf{u})^4)$.  Since $h(w)$ can be transformed into $h(w_1)$ by the application of at most $|w|$ many $P'_1(M_{\infty})$ relators, we conclude that the minimal $P'_1(M_{\infty})$ diagram for $h(w)$ has area $O(\ell(\textbf{u})^4)$.\\

\bibliography{paper}
\bibliographystyle{plain}

\end{document}